\documentclass[12pt, reqno]{amsart}
\usepackage{amsmath, amsthm, amscd, amsfonts, amssymb, graphicx, color, mathrsfs}
\usepackage[bookmarksnumbered, colorlinks, plainpages]{hyperref}
\usepackage{cite}
\usepackage[all]{xy}
\usepackage{slashed}
\usepackage{tikz-cd}
\usepackage{mathabx}
\usepackage{tipa}
\usepackage{soul}
\usepackage{cancel}
\usepackage{ulem}
\usepackage{algorithm}
\usepackage{algpseudocode}
\usepackage{enumitem}
\usepackage{multirow}
\usepackage{float}
\usepackage{verbatim}
\usepackage{amssymb}
\usepackage{extpfeil}

\usepackage{pstricks, pst-plot}
\usepackage{pictex,dcpic}
\usepackage{tikz}
\usepackage{tikz-cd}

\DeclareMathOperator{\Ad}{Ad}
\DeclareMathOperator{\ad}{ad}

\allowdisplaybreaks

\newlength\tindent
\newtheorem{theorem}{Theorem}[section]

\newtheorem{proposition}[theorem]{Proposition}
\newtheorem{corollary}[theorem]{Corollary}

\theoremstyle{definition}
\newtheorem{definition}[theorem]{Definition}

\theoremstyle{remark}
\newtheorem{remark}[theorem]{Remark}
\numberwithin{equation}{section}

\begin{document}
\setcounter{page}{1}

\title[Equigeodesic vectors for homogeneous Riemannian submersions]{Equigeodesic vectors for homogeneous Riemannian submersions}

\author[N. Pereira da Silva]{Neiton Pereira da Silva}
\address{Neiton Pereira da Silva \endgraf 
Universidade Federal de Uberlândia, Instituto de Matemática e Estatística\\ Avenida João Naves de Ávila, 2121\\ CEP 38400-902, Uberlândia - MG, Brazil. \endgraf
{\it E-mail address:} {\rm neiton@ufu.br}
 }
 
\author[B. Grajales]{Brian Grajales}
\address{Brian Grajales \endgraf
Universidade Estadual de Maringá, Departamento de Matem\'{a}tica, Avenida Colombo, 5790, Campus Universitário, 87020-900, Maringá - PR, Brazil 
\endgraf
  {\it E-mail address:} {\rm bdgtriana@uem.br}
  }

\author[L. Grama]{Lino Grama}
\address{Lino Grama \endgraf
IMECC-Unicamp, Departamento de Matem\'{a}tica. Rua S\'{e}rgio Buarque de Holanda \endgraf 651, Cidade Universit\'{a}ria Zeferino Vaz. 13083-859, Campinas - SP, Brazil
\endgraf
{\it E-mail address:} {\rm lino@ime.unicamp.br}
  }


\keywords{Equigeodesics, homogeneous geodesics, Riemannian submersions}
\subjclass[2020]{53C30, 53C22}

\begin{abstract} We study $\pi$-equigeodesic vectors associated with homogeneous fibrations, namely vectors that are geodesic with respect to every homogeneous metric making the projection a Riemannian submersion. We obtain an algebraic criterion characterizing such vectors and apply it to classical flag manifolds and Ledger-Obata spaces. As a framework for this study, given Lie groups $K\subseteq H\subseteq G$ with $H$ and $K$ closed in $G$, and a fixed $G$-invariant metric $g_b$ on $G/H$, we describe the family of $G$-invariant metrics $g$ on $G/K$ for which the natural projection $\pi:(G/K,g)\to(G/H,g_b)$ is a Riemannian submersion. We also give a criterion for the fibers of $\pi$ to be totally geodesic.
\end{abstract}
\maketitle

\allowdisplaybreaks
\section{Introduction}
The study of geodesics is a basic problem in Riemannian geometry. In a general Riemannian manifold, geodesics are described by a second-order differential equation. When the manifold is a homogeneous space, its symmetries allow one to relate this equation to the Lie algebra of the group acting on it. More precisely, if $M=G/K$ is a reductive homogeneous space and the Riemannian metric on $M$ is $G$-invariant, then geodesics which are orbits of one-parameter subgroups of $G$ can be characterized by algebraic conditions on vectors in the Lie algebra. These curves are called {\it homogeneous geodesics}. Equivalently, a vector $X$ in the Lie algebra of $G$ is called a geodesic vector if the curve $t\mapsto \exp(tX)K$ is a geodesic in $G/K$. The existence and classification of homogeneous geodesics have been studied in several classes of homogeneous spaces. In particular, geodesic orbit spaces, for which every geodesic is homogeneous, were studied by Kowalski and Vanhecke in low dimensions \cite{KV}, and later in different families of compact homogeneous spaces and flag manifolds \cite{AlArv,Al,AN,AWZ,Arv,Souris2018}.\\

A related notion is that of an {\it equigeodesic vector}. A vector $X$ in the Lie algebra of $G$ is called equigeodesic if it is a geodesic vector for every $G$-invariant metric on $G/K$. This notion was introduced by Cohen, Grama and Negreiros in \cite{CGN}, where equigeodesics on generalized flag manifolds of type $\mathcal{A}$ were studied. Since then, equigeodesic vectors have been studied in several families of homogeneous spaces, including generalized flag manifolds with two isotropy summands \cite{GN}, with second Betti number one \cite{WZ}, with four isotropy summands \cite{X}, and with five isotropy summands \cite{XT2}, as well as flag manifolds with $G_2$-type $\mathfrak{t}$-roots \cite{S2}.\\


Equigeodesics have also been investigated in other classes of homogeneous spaces, such as Stiefel manifolds and generalized Wallach spaces \cite{S}, as well as in the Finsler setting \cite{XT}. More recently, compact homogeneous spaces with equivalent isotropy summands were considered in \cite{GG}, where the presence of equivalent summands was incorporated into the algebraic criterion for equigeodesic vectors.\\


In this paper, we study homogeneous geodesics in the setting of homogeneous fibrations. Given a Lie group $G$ and closed subgroups $K \subseteq H$, there is a natural smooth fibration
\[
H/K \longrightarrow G/K \stackrel{\pi}{\longrightarrow} G/H.
\]
We fix a $G$-invariant Riemannian metric $g_b$ on the base space $G/H$. Assuming that $G/K$ and $G/H$ are reductive, we choose reductive decompositions $\mathfrak{g}=\mathfrak{h}\oplus\mathfrak{n},\ \mathfrak{h}=\mathfrak{k}\oplus\mathfrak{p}.$ Then $\mathfrak{g}=\mathfrak{k}\oplus\mathfrak{m},$ where $\ \mathfrak{m}=\mathfrak{p}\oplus\mathfrak{n},$ is a reductive decomposition for $G/K$. Hence any $G$-invariant metric on $G/K$ is determined by an $\operatorname{Ad}(K)$-invariant inner product on $\mathfrak{m}$. We characterize the $G$-invariant metrics $g$ on $G/K$ for which
\[
\pi:(G/K,g)\longrightarrow (G/H,g_b)
\]
is a Riemannian submersion. We show that these metrics are parametrized by pairs $(g_f,L)$, where $g_f$ is an $H$-invariant metric on the fiber $H/K$ and $L$ is a $K$-equivariant linear map from $\mathfrak{n}$ to $\mathfrak{p}$. We also give necessary and sufficient conditions on $(g_f,L)$ for the fibers of $\pi$ to be totally geodesic; see Theorem~\ref{theorem:characterization:Riemannian:submersions:totally:geodesic}. We then introduce the notion of a $\pi$-equigeodesic vector. A vector $X\in\mathfrak{m}$ is called $\pi$-equigeodesic with respect to $g_b$ if it is a geodesic vector for every $G$-invariant metric $g$ on $G/K$ such that $\pi:(G/K,g)\to(G/H,g_b)$ is a Riemannian submersion. When $G$ is compact, we write $X=X_{\mathfrak{p}}+X_{\mathfrak{n}}$ according to the decomposition $\mathfrak{m}=\mathfrak{p}\oplus\mathfrak{n}$ and obtain an algebraic criterion for deciding whether $X$ is $\pi$-equigeodesic with respect to $g_b$; see Theorem~\ref{theorem:pi-equigeodesic:criterion}. As an immediate consequence, the fiber component of every $\pi$-equigeodesic vector is an equigeodesic vector on the fiber $H/K$; see Remark~\ref{remark:consequences:criterion}.\\

We apply this criterion to several homogeneous fibrations. First, we consider fibrations whose base and total spaces are classical flag manifolds of $\operatorname{SO}(n)$, $\operatorname{SU}(n)$, and $\operatorname{Sp}(n)$, with $n\geq 5$. We prove that every $\pi$-equigeodesic vector has the form $X_{\mathfrak{p}}+X_{\mathfrak{n}},$ where $X_{\mathfrak{p}}$ is an equigeodesic vector on the fiber and $X_{\mathfrak{n}}$ is a geodesic vector with respect to $g_b$. We also study a homogeneous fibration of $\operatorname{SO}(4)$ for which $\operatorname{Hom}_K(\mathfrak{n},\mathfrak{p})\neq \{0\}$ and the base space has equivalent isotropy summands. Finally, we determine the $\pi$-equigeodesic vectors for a fibration of the Ledger-Obata space $S^{r+1}/\Delta S,$ where $S$ is a simple compact connected Lie group and $r\geq 2.$\\

The paper is organized as follows. In Section~2, we recall the notion of homogeneous geodesic and the algebraic criteria for geodesic and equigeodesic vectors on reductive homogeneous spaces. In Section~3, we study homogeneous Riemannian submersions associated with triples $K\subseteq H\subseteq G$. We characterize the invariant metrics on $G/K$ for which the natural projection onto $G/H$ is a Riemannian submersion, and we give a criterion for the fibers to be totally geodesic. In Section~4, we introduce $\pi$-equigeodesic vectors and obtain an algebraic criterion in terms of the fiber and base components. In Section~5, we apply this criterion to homogeneous fibrations involving classical flag manifolds, a flag manifold of $\operatorname{SO}(4)$ and Ledger-Obata spaces.

\section{Homogeneous geodesics}
Let $G$ be a Lie group, let $K$ be a closed subgroup of $G$ and denote by $\mathfrak{g}$ and $\mathfrak{k}$ their corresponding Lie algebras. Assume that $G/K$ is reductive and fix a reductive decomposition 
\[
\mathfrak{g}=\mathfrak{k}\oplus\mathfrak{m},\ \textnormal{with}\ \mathrm{Ad}(k)\mathfrak{m}=\mathfrak{m},\ \forall k\in K.
\]
Each $X\in\mathfrak{g}$ induces a smooth vector field $X^*$ on $G/K$ defined by 
\begin{equation*}
    X^*(aK):=\left.\frac{d}{dt}\exp(tX)aK\,\right|_{t=0}\ (a\in G),
\end{equation*}
where $\exp:\mathfrak{g}\to G$ is the exponential map of $G.$ In this setting, $\mathfrak{m}$ is identified with the tangent space $T_{eK}(G/K)$ at the identity class via 
$$\mathfrak{m}\ni X\longmapsto X^{*}(eK)\in T_{eK}(G/K).$$
A Riemannian metric $g$ on $G/K$ is said to be $G$-invariant (or $G$-homogeneous) if
\[
g_{aK}((\mathrm{Ad}(a)X)^*(aK),(\mathrm{Ad}(a)Y)^*(aK))=g_{eK}(X^*(eK),Y^*(eK)),
\]
for all $a\in G$ and $X,Y\in\mathfrak{m}.$ It is a well-known fact that the set of $G$-invariant metrics on $G/K$ is in one-to-one correspondence with the set of $\mathrm{Ad}(K)$-invariant inner products on $\mathfrak{m},$ that is, inner products for which $\mathrm{Ad}(k)$ is a linear isometry for each $k\in K.$  Moreover, the correspondence is given by $g\longmapsto\langle\cdot,\cdot\rangle_g,$ where 
\[
\langle X,Y\rangle_g:=g_{eK}(X^{*}(eK),Y^*(eK)).
\]
\begin{definition}
Let $g$ be a $G$-invariant metric on $G/K$. A smooth curve $\gamma:\mathbb{R}\to G/K$ is called {\it homogeneous} if there exist $X\in\mathfrak g$ and $a\in G$ such that
\[
\gamma(t)=\exp(tX)aK.
\]
If, in addition, $\gamma$ is a geodesic in $(G/K,g)$, then \(\gamma\) is called a {\it homogeneous geodesic} with respect to $g.$ A vector $X\in\mathfrak{g}$ is called a {\it geodesic vector} with respect to $g$ if the curve $$\gamma_X(t):=\exp(tX)K$$ is a homogeneous geodesic in $(G/K,g).$ If $X\in\mathfrak{g}$ is a geodesic vector with respect to every $G$-invariant metric on $G/K,$ then $X$ is called an {\it equigeodesic vector.} In this case, $\gamma_X$ is called a {\it homogeneous equigeodesic.}
\end{definition}
\begin{remark}
If $g$ is a $G$-invariant metric on $G/K$, then, for every $a\in G$, the map $\phi_a:G/K\to G/K,\ \phi_a(xK):=axK,$
is an isometry. Moreover, if $\gamma(t)=\exp(tX)aK$ is a homogeneous curve through $aK,$ then $(\phi_{a^{-1}}\circ\gamma)(t)=\exp(t\mathrm{Ad}(a^{-1})X)K.$ Therefore, $\gamma$ is a homogeneous geodesic if and only if the curve $t\mapsto \exp(t\operatorname{Ad}(a^{-1})X)K$ is a homogeneous geodesic through the identity class. Thus, the study of homogeneous geodesics reduces to the study of geodesic vectors.
\end{remark}
For all $X\in\mathfrak g,$ denote by $X_\mathfrak{k}$ and $X_\mathfrak{m}$ the components of $X$ in $\mathfrak{k}$ and $\mathfrak{m},$ respectively. The following theorem, proved by Kowalski and Vanhecke in \cite[Corollary 2.2]{KV}, provides a characterization of geodesic vectors.

\begin{theorem}\label{KV:theorem}
A vector $X\in\mathfrak{g}$ is geodesic with respect to the $G$-invariant metric $g$ if and only if
\begin{equation}\label{KV:criterion}
\langle [X,Y]_\mathfrak{m},X_\mathfrak{m}\rangle_g=0
\end{equation}
for all $Y\in\mathfrak{m}$.
\end{theorem}
As an immediate consequence, $X\in\mathfrak{g}$ is an equigeodesic vector if and only if \eqref{KV:criterion} holds for every $G$-invariant metric $g$ and every $Y\in\mathfrak{m}.$\\

In the case where $G$ is compact, its Lie algebra admits an $\mathrm{Ad}(G)$-invariant inner product, say $(\cdot,\cdot).$ Fixing such an inner product and taking $\mathfrak{m}$ as the $(\cdot,\cdot)$-orthogonal complement of $\mathfrak{k}$ in $\mathfrak{g},$ we have, in particular, that $(\cdot,\cdot)$ defines an $\mathrm{Ad}(K)$-invariant inner product when restricted to $\mathfrak{m}.$ Therefore, for each $G$-invariant metric $g,$ there exists a unique linear map $\Lambda:\mathfrak{m}\to\mathfrak{m}$ such that 
\[
\langle X,Y\rangle_g=(\Lambda X,Y),\ \textnormal{for all}\ X,Y\in\mathfrak{m}.
\]
The map \(\Lambda\) is \((\cdot,\cdot)\)-symmetric, positive definite, and \(\operatorname{Ad}(K)\)-equivariant. Conversely, for every linear map \(\Lambda:\mathfrak m\to\mathfrak m\) with these properties, the formula above defines an \(\operatorname{Ad}(K)\)-invariant inner product on \(\mathfrak m\) and, therefore, a \(G\)-invariant metric $g$ on \(G/K\). In this situation, we refer to $\Lambda$ as the {\it metric operator} associated with $g.$\\

Under the assumption that $G$ is compact, the criterion of Theorem \ref{KV:theorem} can be refined in terms of the metric operator as follows.
\begin{proposition}[\hspace{-0.04em}{\cite[Proposition 3.5]{CGN}}]\label{CGN:theorem}
    If $G$ is compact and $g$ is a $G$-invariant metric on $G/K,$ then $X\in\mathfrak{g}$ is geodesic with respect to $g$ if and only if
    \begin{equation}\label{CGN:criterion}
        [X,\Lambda X_\mathfrak{m}]_\mathfrak{m}=0,
    \end{equation}
    where $\Lambda:\mathfrak{m}\to\mathfrak{m}$ is the metric operator associated with $g.$ As a consequence, $X\in\mathfrak{g}$ is an equigeodesic vector if and only if \eqref{CGN:criterion} holds for every metric operator $\Lambda.$
\end{proposition}
\section{Homogeneous Riemannian submersions}
Let $G$ be a Lie group and let $K,H$ be closed subgroups of $G$ such that $K\subseteq H$ and the homogeneous spaces $G/K$ and $G/H$ are reductive. Then $H/K$ is also reductive. We consider the natural projection
\begin{eqnarray*}
\pi\colon & G/K & \longrightarrow  G/H\\
           & aK & \longmapsto  aH.
\end{eqnarray*}
Equivalently, $G/K$ is the fiber bundle associated with the principal $H$-bundle $G\to G/H,$ with fiber $H/K.$ Let $\mathfrak{g}$ be the Lie algebra of $G$ and let $\mathfrak{k}\subseteq\mathfrak{h}\subseteq\mathfrak{g}$ be the Lie subalgebras associated with $K$ and $H,$ respectively. We choose an $\mathrm{Ad}(H)$-invariant complement $\mathfrak{n}$ of $\mathfrak{h}$ in $\mathfrak{g}$ and an $\mathrm{Ad}(K)$-invariant complement $\mathfrak{p}$ of $\mathfrak{k}$ in $\mathfrak{h}.$ Then 
\[
\mathfrak{g}=\mathfrak{h}\oplus\mathfrak{n},\ \mathfrak{h}=\mathfrak{k}\oplus\mathfrak{p}
\]
and $\mathfrak{m}:=\mathfrak{p}\oplus\mathfrak{n}$ is an $\mathrm{Ad}(K)$-invariant complement of $\mathfrak{k}$ in $\mathfrak{g}.$ For $X\in\mathfrak{m},$ we write
\[
X=X_\mathfrak{p}+X_\mathfrak{n},\ X_\mathfrak{p}\in\mathfrak{p},\ X_\mathfrak{n}\in\mathfrak{n}.
\]
We refer to $X_\mathfrak{p}$ and $X_\mathfrak{n}$ as the {\it fiber component} and the {\it base component} of $X,$ respectively. Under the identifications $T_{eK}(G/K)\simeq \mathfrak{p}\oplus\mathfrak{n}$ and $T_{eH}(G/H)\simeq \mathfrak{n},$ we have
\[
(d\pi)_{eK}:\mathfrak{p}\oplus\mathfrak{n}\to \mathfrak{n},\ 
(d\pi)_{eK}(X)=X_{\mathfrak{n}}.
\] 
Therefore, $\ker(d\pi)_{eK}=\mathfrak{p}.$ In particular, the fiber component of each $X\in\mathfrak{m}$ is precisely its vertical component. The following theorem characterizes, for a fixed $G$-invariant metric on $G/H,$ the $G$-invariant metrics on $G/K$ for which $\pi$ is a Riemannian submersion.
\begin{theorem}\label{theorem:characterization:Riemannian:submersions}
Let $g_b$ be a $G$-invariant metric on $G/H$ and write $\langle\cdot,\cdot\rangle_b:=\langle\cdot,\cdot\rangle_{g_b}.$ The following hold:
\begin{itemize}
\item[$a)$] If $g$ is a $G$-invariant metric on $G/K$ such that the projection
\[
\pi:(G/K,g)\longrightarrow (G/H,g_b)
\]
is a Riemannian submersion, then there exist an $\operatorname{Ad}(K)$-invariant inner product $\langle\cdot,\cdot\rangle_f$ on $\mathfrak{p}$ and a $K$-equivariant linear map $L\in \mathrm{Hom}_K(\mathfrak{n},\mathfrak{p})$ such that, for all $X,Y\in\mathfrak{p}\oplus\mathfrak{n},$
\begin{equation}\label{general:invariant:product}
\langle X,Y\rangle_g
=
\langle X_\mathfrak{p}-LX_\mathfrak{n},Y_\mathfrak{p}-LY_\mathfrak{n}\rangle_f
+
\langle X_\mathfrak{n},Y_\mathfrak{n}\rangle_b.
\end{equation}

\item[$b)$] Conversely, for any $\operatorname{Ad}(K)$-invariant inner product $\langle\cdot,\cdot\rangle_f$ on $\mathfrak{p}$ and any $K$-equivariant linear map $L\in \mathrm{Hom}_K(\mathfrak{n},\mathfrak{p}),$ formula \eqref{general:invariant:product} defines an $\operatorname{Ad}(K)$-invariant inner product on $\mathfrak{p}\oplus\mathfrak{n}$ such that, for the corresponding $G$-invariant metric $g,$ the projection
\[
\pi:(G/K,g)\longrightarrow (G/H,g_b)
\]
is a Riemannian submersion.

\end{itemize}
\end{theorem}
\begin{proof} Let us prove $a).$ Let $\mathfrak{p}^\perp\subseteq\mathfrak{p}\oplus\mathfrak{n}$ be the $\langle\cdot,\cdot\rangle_g$-orthogonal complement of $\mathfrak{p}$ in $\mathfrak{p}\oplus\mathfrak{n}.$ If $\pi$ is a Riemannian submersion, then the restriction
\[
(d\pi)_{eK}\big{|}_{\mathfrak{p}^\perp}:(\mathfrak{p}^\perp,\langle\cdot,\cdot\rangle_g)\to(\mathfrak{n},\langle\cdot,\cdot\rangle_b)
\]
is a linear isometry. Define $L:\mathfrak{n}\to\mathfrak{p}$ by 
\begin{equation*}
    LY:=\left[(d\pi)_{eK}\big{|}_{\mathfrak{p}^\perp}\right]^{-1}(Y)-Y.
\end{equation*}
For a given $Y\in \mathfrak{n},$ since $(d\pi)_{eK}\left(\left[(d\pi)_{eK}\big{|}_{\mathfrak{p}^\perp}\right]^{-1}(Y)\right)=Y$ and $(d\pi)_{eK}(Y)=Y,$ we have $(d\pi)_{eK}(LY)=0,$ so $LY\in\mathrm{ker}(d\pi)_{eK}=\mathfrak{p}.$ This proves that $L$ is well-defined. In addition, it is clear that $L$ is linear. To see that $L$ is $K$-equivariant, observe that 
\begin{align*}
    \mathfrak{p}^{\perp}=\left\{\left[(d\pi)_{eK}\big{|}_{\mathfrak{p}^\perp}\right]^{-1}(Y):Y\in\mathfrak{n}\right\}=\left\{LY+Y:Y\in\mathfrak{n}\right\}.
\end{align*}
Since $\langle\cdot,\cdot\rangle_g$ is $\mathrm{Ad}(K)$-invariant and $\mathfrak{p}$ is $\mathrm{Ad}(K)$-invariant, $\mathfrak{p}^{\perp}$ is $\mathrm{Ad}(K)$-invariant. Thus, given $k\in K$ and $Y\in\mathfrak{n},$ it follows that $\mathrm{Ad}(k)\left(LY+Y\right)=\mathrm{Ad}(k)LY+\mathrm{Ad}(k)Y\in\mathfrak{p}^\perp.$ Hence, $\mathrm{Ad}(k)LY+\mathrm{Ad}(k)Y=LY_1+Y_1,$ for some $Y_1\in\mathfrak{n}.$ Since \(\mathfrak p\) and \(\mathfrak n\) are \(\mathrm{Ad}(K)\)-invariant, we have \(\mathrm{Ad}(k)LY,LY_1\in\mathfrak p\) and \(\mathrm{Ad}(k)Y,Y_1\in\mathfrak n\). By the uniqueness of the decomposition in \(\mathfrak p\oplus\mathfrak n\), we get $\mathrm{Ad}(k)LY=LY_1$ and $\mathrm{Ad}(k)Y=Y_1.$ Hence, $\mathrm{Ad}(k)LY=L\mathrm{Ad}(k)Y,$ that is, $L$ is $K$-equivariant. Now, for $X_1,X_2\in\mathfrak{p},$ define $\langle X_1,X_2\rangle_f:=\langle X_1,X_2\rangle_g.$ Since $\langle\cdot,\cdot\rangle_g$ is $\mathrm{Ad}(K)$-invariant and $\mathfrak{p}$ is $\mathrm{Ad}(K)$-invariant, $\langle\cdot,\cdot\rangle_f$ is $\mathrm{Ad}(K)$-invariant. Since $(d\pi)_{eK}\big{|}_{\mathfrak{p}^\perp}$ is a linear isometry, we have
 \begin{align*}
     \langle LY_1+Y_1,LY_2+Y_2\rangle_g&=\langle(d\pi)_{eK}(LY_1+Y_1),(d\pi)_{eK}(LY_2+Y_2)\rangle_{b}\\
     &=\langle Y_1,Y_2\rangle_b,
 \end{align*}
 for all $Y_1,Y_2\in\mathfrak{n}.$ For $X,Y\in\mathfrak{p}\oplus\mathfrak{n},$ we have 
 \begin{align*}
     &X=X_\mathfrak{p}+X_\mathfrak{n}=(X_\mathfrak{p}-LX_{\mathfrak{n}})+(LX_\mathfrak{n}+X_\mathfrak{n})\\
     &Y=Y_\mathfrak{p}+Y_\mathfrak{n}=(Y_\mathfrak{p}-LY_{\mathfrak{n}})+(LY_\mathfrak{n}+Y_\mathfrak{n}),
 \end{align*}
 with $X_\mathfrak{p}-LX_{\mathfrak{n}},Y_\mathfrak{p}-LY_{\mathfrak{n}}\in\mathfrak{p}$ and $LX_\mathfrak{n}+X_\mathfrak{n},LY_\mathfrak{n}+Y_\mathfrak{n}\in\mathfrak{p}^\perp.$ Therefore
 \begin{align*}
     \langle X,Y\rangle_g&=\langle X_\mathfrak{p}-LX_\mathfrak{n},Y_\mathfrak{p}-LY_\mathfrak{n}\rangle_g+\langle LX_\mathfrak{n}+X_\mathfrak{n},LY_\mathfrak{n}+Y_\mathfrak{n}\rangle_g\\
     &=\langle X_\mathfrak{p}-LX_\mathfrak{n},Y_\mathfrak{p}-LY_\mathfrak{n}\rangle_f+\langle X_\mathfrak{n},Y_\mathfrak{n}\rangle_b.
 \end{align*}
 This proves $a)$.\\

\noindent Conversely, let \(\langle\cdot,\cdot\rangle_f\) be an \(\operatorname{Ad}(K)\)-invariant inner product on \(\mathfrak p\) and let \(L\in\mathrm{Hom}_K(\mathfrak n,\mathfrak p)\). Define \(\langle\cdot,\cdot\rangle_g\) on \(\mathfrak p\oplus\mathfrak n\) by \eqref{general:invariant:product}. This is clearly bilinear and symmetric. Moreover, if \(X=X_\mathfrak p+X_\mathfrak n\), then
\[
\langle X,X\rangle_g
=
\langle X_\mathfrak p-LX_\mathfrak n,X_\mathfrak p-LX_\mathfrak n\rangle_f
+
\langle X_\mathfrak n,X_\mathfrak n\rangle_b.
\]
Thus \(\langle X,X\rangle_g\geq 0\). If \(\langle X,X\rangle_g=0\), then \(X_\mathfrak n=0\) and \(X_\mathfrak p-LX_\mathfrak n=0\). Hence \(X_\mathfrak n=0\) and \(X_\mathfrak p=0\). Therefore \(\langle\cdot,\cdot\rangle_g\) is an inner product. We now prove that it is \(\operatorname{Ad}(K)\)-invariant. Let \(k\in K\) and \(X,Y\in\mathfrak p\oplus\mathfrak n\). Since \(\mathfrak p\) and \(\mathfrak n\) are \(\operatorname{Ad}(K)\)-invariant, we have
\[
(\operatorname{Ad}(k)X)_\mathfrak p=\operatorname{Ad}(k)X_\mathfrak p,\ (\operatorname{Ad}(k)X)_\mathfrak n=\operatorname{Ad}(k)X_\mathfrak n.
\]
Since \(L\) is \(K\)-equivariant, $L(\operatorname{Ad}(k)X_\mathfrak n)=\operatorname{Ad}(k)LX_\mathfrak n.$ Therefore
\[
(\operatorname{Ad}(k)X)_\mathfrak p-L(\operatorname{Ad}(k)X)_\mathfrak n
=
\operatorname{Ad}(k)(X_\mathfrak p-LX_\mathfrak n).
\]
Using the \(\operatorname{Ad}(K)\)-invariance of \(\langle\cdot,\cdot\rangle_f\) and the \(\operatorname{Ad}(K)\)-invariance of \(\langle\cdot,\cdot\rangle_b\), we get
\[
\begin{aligned}
\langle \operatorname{Ad}(k)X,\operatorname{Ad}(k)Y\rangle_g
=&\,\langle \operatorname{Ad}(k)(X_\mathfrak p-LX_\mathfrak n),
\operatorname{Ad}(k)(Y_\mathfrak p-LY_\mathfrak n)\rangle_f\\
&+
\langle \operatorname{Ad}(k)X_\mathfrak n,\operatorname{Ad}(k)Y_\mathfrak n\rangle_b\\
=&\,\langle X_\mathfrak p-LX_\mathfrak n,Y_\mathfrak p-LY_\mathfrak n\rangle_f
+
\langle X_\mathfrak n,Y_\mathfrak n\rangle_b\\
=&\,\langle X,Y\rangle_g.
\end{aligned}
\]
Hence \(\langle\cdot,\cdot\rangle_g\) is \(\operatorname{Ad}(K)\)-invariant, and therefore it defines a \(G\)-invariant metric \(g\) on \(G/K\). It remains to prove that \(\pi:(G/K,g)\to(G/H,g_b)\) is a Riemannian submersion. Observe that, for every \(Y\in\mathfrak n\) and every \(Z\in\mathfrak p\),
\[
\langle Z,LY+Y\rangle_g=\langle Z,LY-LY\rangle_f
+
\langle 0,Y\rangle_b
=0.
\]
Hence $\{LY+Y:Y\in\mathfrak{n}\}
\subseteq \mathfrak{p}^\perp.$
Since both spaces have dimension \(\dim\mathfrak n\), we get $\mathfrak{p}^\perp=\{LY+Y:Y\in\mathfrak n\}.$
Additionally, for \(Y_1,Y_2\in\mathfrak n\),
\[
\begin{aligned}
\langle LY_1+Y_1,LY_2+Y_2\rangle_g
&=\langle LY_1-LY_1,LY_2-LY_2\rangle_f+\langle Y_1,Y_2\rangle_b\\
&=
\langle Y_1,Y_2\rangle_b.
\end{aligned}
\]
Since $(d\pi)_{eK}(LY+Y)=Y,$
the restriction
\[
(d\pi)_{eK}\big{|}_{\mathfrak p^\perp}:(\mathfrak p^\perp,\langle\cdot,\cdot\rangle_g)\to(\mathfrak n,\langle\cdot,\cdot\rangle_b)
\]
is a linear isometry. Therefore, \(\pi\) is a Riemannian submersion at \(eK\). Since \(g\) and \(g_b\) are \(G\)-invariant and \(\pi\) is \(G\)-equivariant, the same holds at every point of \(G/K\). This proves \(b)\).
\end{proof}
We now give an algebraic criterion for the fibers of $\pi$ to be totally geodesic.
\begin{theorem}\label{theorem:characterization:Riemannian:submersions:totally:geodesic}
Let $g_b$ be a $G$-invariant metric on $G/H$ and write $\langle\cdot,\cdot\rangle_b:=\langle\cdot,\cdot\rangle_{g_b}.$ The following hold:

\begin{itemize}
    \item[$a)$] For $g$ as in \eqref{general:invariant:product}, the fibers of $\pi:(G/K,g)\longrightarrow (G/H,g_b)$ are totally geodesic if and only if
\begin{equation}\label{condition:totally:geodesic}
\langle [LY,U]_{\mathfrak{p}}-L[Y,U],V\rangle_f
+
\langle U,[LY,V]_{\mathfrak{p}}-L[Y,V]\rangle_f
=0,
\end{equation}
for all $U,V\in\mathfrak{p}$ and $Y\in\mathfrak{n}.$ 

\item[$b)$]  For each \(U\in \mathfrak{p}\), consider the following diagram
\[
\begin{tikzcd}
\mathfrak{n} \arrow[r,"\ad(U)"] \arrow[d, "\ad(U)\circ L"']
& \mathfrak{n} \arrow[d, "L"] \\
\mathfrak{h} \arrow[r,"\Pi_{\mathfrak{p}}"]
& \mathfrak{p}
\end{tikzcd}
\]
where \(\Pi_{\mathfrak{p}}:\mathfrak{h}\to \mathfrak{p}\), 
\(\Pi_{\mathfrak{p}}(X)=X_{\mathfrak{p}}\), denotes the projection onto 
\(\mathfrak{p}\). If this diagram commutes for every \(U\in\mathfrak{p}\), i.e., $\Pi_{\mathfrak{p}}\circ \ad(U)\circ L=L\circ \ad(U),$ then the fibers of the Riemannian submersion \(\pi\) are totally geodesic.

\item[$c)$] If $L\colon \mathfrak{n}\to \mathfrak{p}$ satisfies 
\begin{equation}\label{H-equivariance}
    \operatorname{Ad}(h)L=L\operatorname{Ad}(h)\big{|}_{\mathfrak{n}},\ \textnormal{for all}\ h\in H,
\end{equation}
then, for $g$ as in \eqref{general:invariant:product}, the fibers of $\pi:(G/K,g)\longrightarrow (G/H,g_b)$ are totally geodesic.

\item[$d)$] If $H$ is connected and $L\colon \mathfrak{n}\to \mathfrak{p}$ satisfies 
\begin{equation}\label{h-equivariance}
    \operatorname{ad}(U)L=L\operatorname{ad}(U)\big{|}_{\mathfrak{n}},\ \textnormal{for all}\ U\in\mathfrak{h},
\end{equation} 
then, for $g$ as in \eqref{general:invariant:product}, the fibers of $\pi:(G/K,g)\longrightarrow (G/H,g_b)$ are totally geodesic.
 
\end{itemize}
\begin{proof}
$a)$ Since $\pi$ is $G$-equivariant and $g$ is
$G$-invariant, it is enough to consider the fiber through the origin,
$\pi^{-1}(eH)=H/K$. At $eK$, this fiber has tangent space
$\mathfrak{p}$. Moreover, we already showed that $\mathfrak{p}^{\perp}=\{LY+Y:Y\in\mathfrak{n}\}.$ Therefore, $H/K$ is totally geodesic if and only if
\[
\langle \nabla_UV,LY+Y\rangle_g=0
\]
for all $U,V\in\mathfrak{p}$ and $Y\in\mathfrak{n}$, where $\nabla$
denotes the Levi-Civita connection of $(G/K,g)$. By \cite[Lemma 7.27]{B}, for $U,V\in\mathfrak{p}$ and $Y\in\mathfrak{n}$ we have
\[
2\langle \nabla_UV,LY+Y\rangle_g=\langle [V,U]_{\mathfrak{m}},LY+Y\rangle_g+\langle [LY+Y,U]_{\mathfrak{m}},V\rangle_g+\langle [LY+Y,V]_{\mathfrak{m}},U\rangle_g.
\]
Since $U,V\in\mathfrak{p}\subseteq\mathfrak{h}$, we have $[U,V]\in\mathfrak{h}=\mathfrak{k}\oplus\mathfrak{p}.$ Hence $[U,V]_{\mathfrak{m}}\in\mathfrak{p}$. Since $LY+Y\in\mathfrak{p}^{\perp}$, it follows that 
\[
\langle [U,V]_{\mathfrak{m}},LY+Y\rangle_g=0.
\]
On the other hand, since $U,LY\in\mathfrak{p}$ and $Y\in\mathfrak{n}$, then $[Y,U]\in\mathfrak{n}$ and $[LY,U]\in\mathfrak{h}=\mathfrak{k}\oplus\mathfrak{p}.$ Therefore
\[
[LY+Y,U]_{\mathfrak{m}}
=
[LY,U]_{\mathfrak{p}}+[Y,U].
\]
Using \eqref{general:invariant:product}, we get
\[
\langle [LY+Y,U]_{\mathfrak{m}},V\rangle_g=\langle [LY,U]_{\mathfrak{p}}+[Y,U],V\rangle_g=\langle [LY,U]_{\mathfrak{p}}-L[Y,U],V\rangle_f.
\]
Analogously,
\[
\langle [LY+Y,V]_{\mathfrak{m}},U\rangle_g
=
\langle [LY,V]_{\mathfrak{p}}-L[Y,V],U\rangle_f.
\]
Consequently,
\[
2\langle \nabla_UV,LY+Y\rangle_g=\langle [LY,U]_{\mathfrak{p}}-L[Y,U],V\rangle_f+
\langle [LY,V]_{\mathfrak{p}}-L[Y,V],U\rangle_f.
\]
Thus, $H/K$ is totally geodesic if and only if
\[
\langle [LY,U]_{\mathfrak{p}}-L[Y,U],V\rangle_f
+
\langle U,[LY,V]_{\mathfrak{p}}-L[Y,V]\rangle_f
=0
\]
for all $U,V\in\mathfrak{p}$ and $Y\in\mathfrak{n}$. This proves $a)$.\\

$b)$ It follows immediately from equation (\ref{condition:totally:geodesic}).\\

$c)$ If $L$ satisfies \eqref{H-equivariance}, then, for any $h\in H$ and $Y\in \mathfrak{n}$, $\operatorname{Ad}(h)LY=L(\operatorname{Ad}(h)Y)\in \mathfrak{p}.$ Let $U\in \mathfrak{p}.$ Since $U\in \mathfrak{h}$, we have $\exp(tU)\in H$. Thus
\[
\operatorname{Ad}(\exp(tU))LY=L(\operatorname{Ad}(\exp(tU))Y).
\]
By \cite[p.~123]{book SM}, $d(\operatorname{Ad})_e(U)=\operatorname{ad}(U)$. Therefore, differentiating the previous identity at $t=0$, we obtain $[U,LY]=L[U,Y].$ Equivalently, $[LY,U]=L[Y,U]$. Since $L[Y,U]\in \mathfrak{p}$, it follows that $[LY,U]_{\mathfrak{p}}=L[Y,U].$ Thus $[LY,U]_{\mathfrak{p}}-L[Y,U]=0$. Analogously, for every $V\in \mathfrak{p}$, $[LY,V]_{\mathfrak{p}}-L[Y,V]=0.$ Therefore the condition in item $a)$ is satisfied, and the fibers of $\pi$ are totally geodesic.\\


$d)$ By item $c)$, it is enough to prove that  \eqref{h-equivariance} implies \eqref{H-equivariance}. Let $U\in \mathfrak{h}$ and 
$Y\in \mathfrak{n}$ be any vectors. Then
\begin{eqnarray*}
    \Ad(\exp U)LY &=& e^{\ad(U)}LY\\
    &=& Le^{\ad(U)}Y\\
    &=& L\Ad(\exp U)Y
\end{eqnarray*}
where in the second equality we used the hypothesis $\ad(U)L=L\ad(U)$. Now, since $H$ is connected, for $h\in H$ there exist $U_1, \dots, U_s \in \mathfrak{h}$ such that $h=\exp(U_1)\cdots\exp(U_s),$ then $\Ad(h)LY=L\Ad(h)Y$. This proves $d)$.
\end{proof}
\end{theorem}

\begin{remark} If $\operatorname{Hom}_K(\mathfrak{n},\mathfrak{p})=0$, then necessarily $L=0$. Hence, for a fixed $G$-invariant metric $g_b$ on $G/H$, the $G$-invariant metrics $g$ on $G/K$ for which
\[
\pi:(G/K,g)\longrightarrow (G/H,g_b)
\]
is a Riemannian submersion are in one-to-one correspondence with the $H$-invariant metrics $g_f$ on the fiber $H/K$. In this case,
\[
\langle X,Y\rangle_g
=
\langle X_\mathfrak{p},Y_\mathfrak{p}\rangle_f
+
\langle X_\mathfrak{n},Y_\mathfrak{n}\rangle_b,
\]
for all $X,Y\in\mathfrak{m}$. For metrics of this form, the condition \eqref{condition:totally:geodesic} is automatically satisfied. Hence, the fibers of $\pi$ are totally geodesic in $G/K$, as also follows from a result of Bérard-Bergery \cite[p.~60]{BB}. \qed

\end{remark}

\noindent Assume now that \(G\) is compact and fix an \(\operatorname{Ad}(G)\)-invariant inner product \((\cdot,\cdot)\) on \(\mathfrak g\). We take the decompositions
\[
\mathfrak g=\mathfrak h\oplus\mathfrak n,\ \mathfrak h=\mathfrak k\oplus\mathfrak p
\]
to be orthogonal with respect to \((\cdot,\cdot)\). Let \(g\), \(g_b\), and \(g_f\) be invariant metrics on \(G/K\), \(G/H\), and \(H/K\), respectively, where \(g\) and \(g_b\) are \(G\)-invariant and \(g_f\) is \(H\)-invariant. Let \(\Lambda:\mathfrak p\oplus\mathfrak n\to\mathfrak p\oplus\mathfrak n\), \(\Lambda_b:\mathfrak n\to\mathfrak n\), and \(\Lambda_f:\mathfrak p\to\mathfrak p\) be the corresponding metric operators with respect to \((\cdot,\cdot)\). In this setting, the previous theorems can be reformulated in terms of metric operators as follows.

\begin{corollary}\label{corollary:compact:general:metric:operator}
Fix the metric operator \(\Lambda_b:\mathfrak n\to\mathfrak n\) associated with the \(G\)-invariant metric \(g_b\) on \(G/H\). Then the \(G\)-invariant metrics \(g\) on \(G/K\) for which
\[
\pi:(G/K,g)\longrightarrow (G/H,g_b)
\]
is a Riemannian submersion are in one-to-one correspondence with the pairs \((\Lambda_f,L)\), where \(\Lambda_f:\mathfrak p\to\mathfrak p\) is the metric operator of some \(H\)-invariant metric \(g_f\) on \(H/K\) and \(L\in\operatorname{Hom}_K(\mathfrak n,\mathfrak p)\). For such a pair $(\Lambda_f,L),$ the corresponding $G$-invariant metric $g$ on $G/K$ has metric operator $\Lambda:\mathfrak{p}\oplus\mathfrak{n}\to\mathfrak{p}\oplus\mathfrak{n}$ given by
\begin{equation}\label{general:metric:operator}
\Lambda X
=
\Lambda_f(X_\mathfrak{p}-LX_\mathfrak{n})
+
\Lambda_bX_\mathfrak{n}
-
L^*\Lambda_f(X_\mathfrak{p}-LX_\mathfrak{n}),
\end{equation}
where \(L^*:\mathfrak p\to\mathfrak n\) denotes the adjoint of \(L\) with respect to \((\cdot,\cdot)\). Moreover, the fibers of $\pi$ are totally geodesic if and only if
\begin{equation}\label{condition:totally:geodesic:operators}
\Lambda_f A_Y^L+\left(A_Y^L\right)^*\Lambda_f=0
\end{equation}
for every $Y\in\mathfrak{n}$, where $A_Y^L:\mathfrak{p}\longrightarrow\mathfrak{p},\ A_Y^L(U)=[LY,U]_{\mathfrak{p}}-L[Y,U],$
and $\left(A_Y^L\right)^*$ denotes the adjoint of $A_Y^L$ with respect to
$(\cdot,\cdot)$ restricted to $\mathfrak{p}$.
\end{corollary}
\begin{proof}
Observe that \eqref{general:invariant:product} can be written in terms of metric operators as
\begin{equation}\label{general:invariant:product:compact:case}
(\Lambda X,Y)=(\Lambda_f(X_\mathfrak{p}-LX_\mathfrak{n}),Y_\mathfrak{p}-LY_\mathfrak{n})+(\Lambda_bX_\mathfrak{n},Y_\mathfrak{n}),\ \forall X,Y\in\mathfrak{m}.
\end{equation}
On the other hand,
\[
\begin{aligned}
&
(\Lambda_f(X_\mathfrak{p}-LX_\mathfrak{n}),Y_\mathfrak{p}-LY_\mathfrak{n})
+
(\Lambda_bX_\mathfrak{n},Y_\mathfrak{n})\\
=&
(\Lambda_f(X_\mathfrak{p}-LX_\mathfrak{n}),Y_\mathfrak{p})
-
(\Lambda_f(X_\mathfrak{p}-LX_\mathfrak{n}),LY_\mathfrak{n})
+
(\Lambda_bX_\mathfrak{n},Y_\mathfrak{n})\\
=&
(\Lambda_f(X_\mathfrak{p}-LX_\mathfrak{n}),Y_\mathfrak{p})
+
(\Lambda_bX_\mathfrak{n}-L^*\Lambda_f(X_\mathfrak{p}-LX_\mathfrak{n}),Y_\mathfrak{n})\\
=&
\left(
\Lambda_f(X_\mathfrak{p}-LX_\mathfrak{n})
+
\Lambda_bX_\mathfrak{n}
-
L^*\Lambda_f(X_\mathfrak{p}-LX_\mathfrak{n}),
Y
\right).
\end{aligned}
\]
Therefore, \eqref{general:invariant:product:compact:case} holds for all $X,Y\in\mathfrak{m}$ if and only if 
\[
\Lambda X
=
\Lambda_f(X_\mathfrak{p}-LX_\mathfrak{n})
+
\Lambda_bX_\mathfrak{n}
-
L^*\Lambda_f(X_\mathfrak{p}-LX_\mathfrak{n}),
\]
for all $X\in\mathfrak{m}.$ The first assertion now follows from Theorem~\ref{theorem:characterization:Riemannian:submersions}.\\

On the other hand, by Theorem~\ref{theorem:characterization:Riemannian:submersions:totally:geodesic} $a)$, the
fibers of $\pi$ are totally geodesic if and only if
\[
\langle A_Y^L(U),V\rangle_f+\langle U,A_Y^L(V)\rangle_f=0
\]
for all $U,V\in\mathfrak{p}$ and $Y\in\mathfrak{n}$. Since $\langle U,V\rangle_f=(\Lambda_f U,V),$ this condition is equivalent to
\begin{align*}
(\Lambda_f A_Y^L(U),V)+(\Lambda_f U,A_Y^L(V))=0\Longleftrightarrow\left((\Lambda_f A_Y^L+(A_Y^L)^*\Lambda_f)U,V\right)=0,
\end{align*}
for all $U,V\in\mathfrak{p}$ and $Y\in\mathfrak{n}$. Therefore, the fibers of $\pi$ are totally geodesic if and only if \eqref{condition:totally:geodesic:operators} holds for all $Y\in\mathfrak{n}.$
\end{proof}

\section{$\pi$-Equigeodesic vectors}
In this section, we fix a $G$-invariant metric $g_b$ on $G/H$ and use the description obtained in the previous section to study geodesic vectors on $G/K$ in terms of their fiber and base components.
\begin{theorem}\label{theorem:geodesic:criterion:fibration}
Let \(g\) be a \(G\)-invariant metric on \(G/K\) such that
\[
\pi:(G/K,g)\longrightarrow (G/H,g_b)
\]
is a Riemannian submersion. Assume that $\langle\cdot,\cdot\rangle_g$ is given by \eqref{general:invariant:product}. Then \(X\in\mathfrak p\oplus\mathfrak n\) is a geodesic vector on \((G/K,g)\) if and only if
\begin{equation}\label{geodesic:pi-lemma}
\left\{\begin{array}{l}
\left\langle
[X_\mathfrak{p},Z]_\mathfrak{p}
-
L[X_\mathfrak{n},Z],
X_\mathfrak{p}-LX_\mathfrak{n}
\right\rangle_f=0,
\ \forall\,Z\in\mathfrak p,\\[0.8em]
\begin{aligned}
&
\left\langle
[X_\mathfrak{n},Y]_\mathfrak{p}
-
L\big([X_\mathfrak{p},Y]+[X_\mathfrak{n},Y]_\mathfrak{n}\big),
X_\mathfrak{p}-LX_\mathfrak{n}
\right\rangle_f\\
&+
\left\langle
[X_\mathfrak{p},Y]
+
[X_\mathfrak{n},Y]_\mathfrak{n},
X_\mathfrak{n}
\right\rangle_b=0,\ \forall\,Y\in\mathfrak n.
\end{aligned}
\end{array}
\right.
\end{equation}
\noindent In particular, if $G$ is compact and $(\cdot,\cdot)$ is an $\operatorname{Ad}(G)$-invariant inner product on $\mathfrak{g},$ then \eqref{geodesic:pi-lemma} is equivalent to
\begin{equation}\label{geodesic:pi-lemma:compact}
\left\{
\begin{array}{l}
\left[
X_\mathfrak{p},
\Lambda_f(X_\mathfrak{p}-LX_\mathfrak{n})
\right]_\mathfrak{p}
-
\left[
X_\mathfrak{n},
L^*\Lambda_f(X_\mathfrak{p}-LX_\mathfrak{n})
\right]_\mathfrak{p}
=0,
\\[0.8em]
\begin{aligned}
&
\left[
X_\mathfrak{p},
\Lambda_bX_\mathfrak{n}
-
L^*\Lambda_f(X_\mathfrak{p}-LX_\mathfrak{n})
\right]\\
&+
\left[
X_\mathfrak{n},
\Lambda_f(X_\mathfrak{p}-LX_\mathfrak{n})
\right]
+
\left[
X_\mathfrak{n},
\Lambda_bX_\mathfrak{n}
-
L^*\Lambda_f(X_\mathfrak{p}-LX_\mathfrak{n})
\right]_\mathfrak{n}=0,
\end{aligned}
\end{array}
\right.
\end{equation}
where \(L^*:\mathfrak p\to\mathfrak n\) denotes the adjoint of \(L\) with respect to \((\cdot,\cdot)\).
\end{theorem}
\begin{proof}
By Theorem~\ref{KV:theorem}, \(X=X_\mathfrak{p}+X_\mathfrak{n}\in\mathfrak{p}\oplus\mathfrak{n}\) is a geodesic vector on \((G/K,g)\) if and only if
\[
\langle [X,V]_{\mathfrak p\oplus\mathfrak n},X\rangle_g=0,
\ \forall\,V\in\mathfrak p\oplus\mathfrak n.
\]
Write \(V=Z+Y\), with \(Z\in\mathfrak{p}\) and \(Y\in\mathfrak{n}\). Since \(\mathfrak{n}\) is \(\operatorname{Ad}(H)\)-invariant and \(\mathfrak{p}\subseteq\mathfrak{h}\), we have $[\mathfrak{p},\mathfrak{n}]\subseteq\mathfrak n.$ Also, since \(\mathfrak{h}=\mathfrak{k}\oplus\mathfrak {p}\), we have $[\mathfrak{p},\mathfrak {p}]\subseteq\mathfrak{h}.$
Therefore
\[
[X,Z]_{\mathfrak p\oplus\mathfrak n}
=
[X_\mathfrak p,Z]_\mathfrak p+[X_\mathfrak n,Z],
\]
where \([X_\mathfrak p,Z]_\mathfrak p\in\mathfrak p\) and \([X_\mathfrak n,Z]\in\mathfrak n\). By \eqref{general:invariant:product},
\[
\begin{aligned}
\langle [X,Z]_{\mathfrak p\oplus\mathfrak n},X\rangle_g&=\left\langle
[X_\mathfrak p,Z]_\mathfrak p-L[X_\mathfrak n,Z],X_\mathfrak p-LX_\mathfrak n\right\rangle_f+\left\langle [X_\mathfrak n,Z],X_\mathfrak n\right\rangle_b.
\end{aligned}
\]
Since $Z\in\mathfrak{p}\subseteq \mathfrak{h}$ and $\langle\cdot,\cdot\rangle_b$ is $\mathrm{Ad}(H)$-invariant, then 
\begin{align*}
    2\langle[X_\mathfrak{n},Z],X_\mathfrak{n}\rangle_b=-\left(\langle[Z,X_\mathfrak{n}],X_\mathfrak{n}\rangle_b+\langle X_\mathfrak{n},[Z,X_\mathfrak{n}]\rangle_b\right)=0.
\end{align*} Hence $\langle [X,Z]_{\mathfrak{p}\oplus\mathfrak{n}},X\rangle_g=0$ for all $Z\in\mathfrak{p}$
is equivalent to
\[
\left\langle
[X_\mathfrak{p},Z]_\mathfrak{p}
-
L[X_\mathfrak{n},Z],
X_\mathfrak{p}-LX_\mathfrak{n}
\right\rangle_f=0
\ \forall\,Z\in\mathfrak p.
\]
\noindent On the other hand, $[X,Y]_{\mathfrak{p}\oplus\mathfrak{n}}=[X_\mathfrak{n},Y]_\mathfrak{p}+[X_\mathfrak{p},Y]+[X_\mathfrak{n},Y]_\mathfrak{n},$
because \([X_\mathfrak p,Y]\in\mathfrak n\). Therefore
\[
\begin{aligned}
\langle [X,Y]_{\mathfrak{p}\oplus\mathfrak{n}},X\rangle_g
={}&
\left\langle[X_\mathfrak{n},Y]_\mathfrak{p}-L\big([X_\mathfrak{p},Y]+[X_\mathfrak {n},Y]_\mathfrak{n}\big),
X_\mathfrak{p}-LX_\mathfrak{n}
\right\rangle_f\\
&+\left\langle
[X_\mathfrak{p},Y]+[X_\mathfrak{n},Y]_\mathfrak{n},
X_\mathfrak{n}
\right\rangle_b.
\end{aligned}
\]
Thus, $\langle [X,Y]_{\mathfrak{p}\oplus\mathfrak{n}},X\rangle_g=0,\ \forall\,Y\in\mathfrak{n}$ if and only if 
$$\begin{aligned}
&
\left\langle[X_\mathfrak{n},Y]_\mathfrak{p}-L\big([X_\mathfrak{p},Y]+[X_\mathfrak {n},Y]_\mathfrak{n}\big),
X_\mathfrak{p}-LX_\mathfrak{n}
\right\rangle_f\\
&+\left\langle
[X_\mathfrak{p},Y]+[X_\mathfrak{n},Y]_\mathfrak{n},
X_\mathfrak{n}
\right\rangle_b=0.
\end{aligned}$$
We have proven that the condition $\langle[X,V]_{\mathfrak{p}\oplus\mathfrak{n}},X\rangle_g=0$ for all $V\in\mathfrak{p}\oplus\mathfrak{n}$ is equivalent to \eqref{geodesic:pi-lemma}. Now, assume that \(G\) is compact. Then \(X\in\mathfrak{p}\oplus\mathfrak{n}\) is a geodesic vector if and only if $[X,\Lambda X]_{\mathfrak p\oplus\mathfrak n}=0.$ Using \eqref{general:metric:operator} we have
\[
\begin{aligned}
[X,\Lambda X]
={}&\left[X_\mathfrak{p},
\Lambda_f(X_\mathfrak{p}-LX_\mathfrak{n})\right]\\
&+\left[
X_\mathfrak{p},\Lambda_bX_\mathfrak{n}
-L^*\Lambda_f(X_\mathfrak{p}-LX_\mathfrak{n})\right]\\
&+\left[X_\mathfrak{n},
\Lambda_f(X_\mathfrak{p}-LX_\mathfrak{n})\right]\\
&+\left[X_\mathfrak{n},
\Lambda_bX_\mathfrak{n}-L^*\Lambda_f(X_\mathfrak{p}-LX_\mathfrak{n})
\right].
\end{aligned}
\]
Again, $[\mathfrak{p},\mathfrak{p}]\subseteq\mathfrak{h}$ and $[\mathfrak{p},\mathfrak{n}]\subseteq\mathfrak{n}$ so the $\mathfrak{p}$-component of $[X,\Lambda X]_{\mathfrak{p}\oplus\mathfrak{n}}$ is
\[
\left[
X_\mathfrak{p},
\Lambda_f(X_\mathfrak{p}-LX_\mathfrak{n})
\right]_\mathfrak{p}
+\left[X_\mathfrak{n},
\Lambda_bX_\mathfrak{n}-L^*\Lambda_f(X_\mathfrak{p}-LX_\mathfrak{n})
\right]_\mathfrak{p},
\]
and its \(\mathfrak n\)-component is
\[
\begin{aligned}
&\left[X_\mathfrak{p},\Lambda_bX_\mathfrak{n}-L^*\Lambda_f(X_\mathfrak{p}-LX_\mathfrak{n})
\right]+\left[X_\mathfrak{n},\Lambda_f(X_\mathfrak{p}-LX_\mathfrak{n})
\right]+\left[X_\mathfrak{n},\Lambda_bX_\mathfrak{n}-L^*\Lambda_f(X_\mathfrak{p}-LX_\mathfrak{n})
\right]_\mathfrak{n}.
\end{aligned}
\]
Therefore, \([X,\Lambda X]_{\mathfrak p\oplus\mathfrak n}=0\) is equivalent to 
\begin{equation*}
\left\{
\begin{array}{l}
\left[
X_\mathfrak{p},
\Lambda_f(X_\mathfrak{p}-LX_\mathfrak{n})
\right]_\mathfrak{p}
+
\left[
X_\mathfrak{n},\Lambda_bX_\mathfrak{n}-
L^*\Lambda_f(X_\mathfrak{p}-LX_\mathfrak{n})
\right]_\mathfrak{p}
=0,
\\[0.8em]
\begin{aligned}
&
\left[
X_\mathfrak{p},
\Lambda_bX_\mathfrak{n}
-
L^*\Lambda_f(X_\mathfrak{p}-LX_\mathfrak{n})
\right]\\
&+
\left[
X_\mathfrak{n},
\Lambda_f(X_\mathfrak{p}-LX_\mathfrak{n})
\right]
+
\left[
X_\mathfrak{n},
\Lambda_bX_\mathfrak{n}
-
L^*\Lambda_f(X_\mathfrak{p}-LX_\mathfrak{n})
\right]_\mathfrak{n}=0,
\end{aligned}
\end{array}
\right.
\end{equation*}
It remains to show that $[X_\mathfrak{n},\Lambda_bX_\mathfrak{n}]_\mathfrak{p}=0.$ In fact, for all $Z\in\mathfrak p$ we have
\begin{align*}
([X_\mathfrak{n},\Lambda_bX_\mathfrak{n}]_\mathfrak{p},Z)
&=
([X_\mathfrak{n},\Lambda_bX_\mathfrak{n}],Z)\ \textnormal{(since $\mathfrak{k},\mathfrak{p},\mathfrak{n}$ are pairwise orthogonal)}\\
&=-(\Lambda_bX_{\mathfrak{n}},[X_\mathfrak{n},Z])\\
&=-\langle X_\mathfrak{n},[X_\mathfrak{n},Z]\rangle_b\\
&=0,
\end{align*}
where the last equality follows from the fact that $Z\in\mathfrak{p}\subseteq \mathfrak{h}$ and $\langle\cdot,\cdot\rangle_b$ is $\operatorname{Ad}(H)$-invariant. Thus, $[X_\mathfrak{n},\Lambda_bX_\mathfrak{n}]_\mathfrak{p}=0.$ The proof is complete.
\end{proof}
We now record some consequences of Theorem~\ref{theorem:geodesic:criterion:fibration}. First, we consider the case in which $\operatorname{Hom}_K(\mathfrak{n},\mathfrak{p})=0.$ Although the result follows directly from Theorem~\ref{theorem:geodesic:criterion:fibration}, in this case the criterion takes a simpler form.

\begin{corollary}\label{Corollary:4.2}
Assume that $\operatorname{Hom}_K(\mathfrak{n},\mathfrak{p})=0.$ Let $g$ be a $G$-invariant metric on $G/K$ such that
\[
\pi:(G/K,g)\longrightarrow (G/H,g_b)
\]
is a Riemannian submersion. Then $X\in\mathfrak{p}\oplus\mathfrak{n}$ is a geodesic vector on $(G/K,g)$ if and only if
\begin{equation}\label{criterion:Hom=0}
\left\{
\begin{array}{l}
\left\langle [X_\mathfrak{p},Z]_\mathfrak{p},X_\mathfrak{p}\right\rangle_f=0,\ \forall Z\in\mathfrak{p},\\[0.5em]
\left\langle [X_\mathfrak{n},Y]_\mathfrak{p},X_\mathfrak{p}\right\rangle_f
+
\left\langle [X_\mathfrak{p},Y]+[X_\mathfrak{n},Y]_\mathfrak{n},X_\mathfrak{n}\right\rangle_b=0,\ \forall Y\in\mathfrak{n}.
\end{array}
\right.
\end{equation}
In particular, if $X$ is a geodesic vector on $(G/K,g),$ then its fiber component $X_\mathfrak{p}$ is a geodesic vector on $(H/K,g_f).$ Conversely, if $X\in \mathfrak{p}$ is a geodesic vector on $(H/K,g_f),$ then $X$ is a geodesic vector on $(G/K,g).$  If, in addition, $G$ is compact and $(\cdot,\cdot)$ is an $\operatorname{Ad}(G)$-invariant inner product on $\mathfrak{g},$ then \eqref{criterion:Hom=0} is equivalent to
\begin{equation}\label{criterion:compact:Hom=0}
\left\{
\begin{array}{l}
[X_\mathfrak{p},\Lambda_fX_\mathfrak{p}]_\mathfrak{p}=0,\\[0.5em]
[X_\mathfrak{p},\Lambda_bX_\mathfrak{n}]
+
[X_\mathfrak{n},\Lambda_fX_\mathfrak{p}]
+
[X_\mathfrak{n},\Lambda_bX_\mathfrak{n}]_\mathfrak{n}=0.
\end{array}
\right.
\end{equation}
\end{corollary}

\begin{proof}
Since $\operatorname{Hom}_K(\mathfrak{n},\mathfrak{p})=0,$ we have $L=0.$ Therefore, equations \eqref{criterion:Hom=0} and \eqref{criterion:compact:Hom=0} follow directly by setting $L=0$ in \eqref{geodesic:pi-lemma} and \eqref{geodesic:pi-lemma:compact}, respectively. The first equation of \eqref{criterion:Hom=0} is precisely the condition of Theorem~\ref{KV:theorem} for $X_\mathfrak{p}$ to be a geodesic vector on $(H/K,g_f).$ Finally, since $L=0$, by Theorem~\ref{theorem:characterization:Riemannian:submersions:totally:geodesic} the fibers of $\pi:(G/K,g)\longrightarrow (G/H,g_b)$ are totally geodesic. Therefore, if $X\in\mathfrak{p}$ is a geodesic vector on $(H/K,g_f),$ it is a geodesic vector on $(G/K,g).$
\end{proof}
\begin{remark}
The previous corollary can be used to construct geodesic vectors on $G/K$ from geodesic vectors in the fiber and in the base, which are homogeneous spaces of lower dimension. In fact, assume that \(\operatorname{Hom}_K(\mathfrak n,\mathfrak p)=0\), and let $X=X_\mathfrak{p}+X_\mathfrak{n}\in\mathfrak{p}\oplus\mathfrak{n}.$ Suppose that $X_\mathfrak{p}$ is a geodesic vector on $(H/K,g_f)$ and that $X_\mathfrak{n}$ is a geodesic vector on $(G/H,g_b).$ Then it follows from \eqref{criterion:Hom=0} that $X$ is a geodesic vector on $(G/K,g)$ if and only if
\begin{equation*}
\left\langle [X_\mathfrak{n},Y]_\mathfrak{p},X_\mathfrak{p}\right\rangle_f
+
\left\langle [X_\mathfrak{p},Y],X_\mathfrak{n}\right\rangle_b=0,\ \forall Y\in\mathfrak{n}.
\end{equation*}
If, in addition, $G$ is compact, the same condition can be written in terms of the metric operators as
\begin{equation*}
[X_\mathfrak{p},\Lambda_bX_\mathfrak{n}]
+
[X_\mathfrak{n},\Lambda_fX_\mathfrak{p}]
=0.
\end{equation*}\qed
\end{remark}

We now define $\pi$-equigeodesic vectors.

\begin{definition}
A vector $X\in\mathfrak{p}\oplus\mathfrak{n}$ is called a {\it $\pi$-equigeodesic vector} with respect to $g_b$ if $X$ is a geodesic vector on $(G/K,g)$ for every $G$-invariant metric $g$ on $G/K$ for which
\[
\pi:(G/K,g)\longrightarrow (G/H,g_b)
\]
is a Riemannian submersion. In this case, the curve $\gamma_X(t)=\exp(tX)K$ is called a {\it homogeneous $\pi$-equigeodesic} with respect to $g_b.$
\end{definition}
\begin{theorem}\label{theorem:pi-equigeodesic:criterion}
Assume that $G$ is compact, and let $X\in\mathfrak{p}\oplus\mathfrak{n}.$ Then $X$ is a $\pi$-equigeodesic vector with respect to $g_b$ if and only if, for every $L\in\operatorname{Hom}_K(\mathfrak{n},\mathfrak{p})$ and every metric operator $\Lambda_f:\mathfrak{p}\to\mathfrak{p}$ of an $H$-invariant metric on $H/K$, the following conditions hold:
\begin{equation}\label{7:equations}
\left\{
\begin{array}{l}
[X_\mathfrak{p},\Lambda_fX_\mathfrak{p}]_\mathfrak{p}=[X_\mathfrak{n},L^*\Lambda_fLX_\mathfrak{n}]_\mathfrak{p}=[X_\mathfrak{n},\Lambda_fX_\mathfrak{p}]=0,\\[0.5em]
[X_\mathfrak{p},\Lambda_fLX_\mathfrak{n}]_\mathfrak{p}+[X_\mathfrak{n},L^*\Lambda_fX_\mathfrak{p}]_\mathfrak{p}=0,\\[0.5em]
[X_\mathfrak{p},\Lambda_bX_\mathfrak{n}]
+
[X_\mathfrak{n},\Lambda_bX_\mathfrak{n}]_\mathfrak{n}=0,\\[0.5em]
[X_\mathfrak{n},\Lambda_fLX_\mathfrak{n}]
+
[X_\mathfrak{p},L^*\Lambda_fX_\mathfrak{p}]
+
[X_\mathfrak{n},L^*\Lambda_fX_\mathfrak{p}]_\mathfrak{n}=0,\\[0.5em]
[X_\mathfrak{p},L^*\Lambda_fLX_\mathfrak{n}]
+
[X_\mathfrak{n},L^*\Lambda_fLX_\mathfrak{n}]_\mathfrak{n}=0.
\end{array}
\right.
\end{equation}
In particular, if $\operatorname{Hom}_K(\mathfrak{n},\mathfrak{p})=0$, then $X$ is a $\pi$-equigeodesic vector with respect to $g_b$ if and only if, for every metric operator $\Lambda_f:\mathfrak{p}\to\mathfrak{p}$ of an $H$-invariant metric on $H/K$,
\begin{equation}\label{3:equations}
\left\{
\begin{array}{l}
[X_\mathfrak{p},\Lambda_fX_\mathfrak{p}]_\mathfrak{p}=[X_\mathfrak{n},\Lambda_fX_\mathfrak{p}]=0,\\[0.5em]
[X_\mathfrak{p},\Lambda_bX_\mathfrak{n}]
+
[X_\mathfrak{n},\Lambda_bX_\mathfrak{n}]_\mathfrak{n}=0.
\end{array}
\right.
\end{equation}
\end{theorem}
\begin{proof}
Assume that $X$ is a $\pi$-equigeodesic vector with respect to $g_b,$ let $L\in\operatorname{Hom}_K(\mathfrak{n},\mathfrak{p})$ and let $\Lambda_f:\mathfrak{p}\to \mathfrak{p}$ be a metric operator. Let $t\in\mathbb R$ and $s>0.$ Then $tL\in\operatorname{Hom}_K(\mathfrak n,\mathfrak p)$ and $s\Lambda_f$ is again a metric operator. By Theorem~\ref{theorem:geodesic:criterion:fibration}, we have 
\begin{equation*}
\left\{
\begin{array}{l}
\left[
X_\mathfrak{p},
s\Lambda_f(X_\mathfrak{p}-tLX_\mathfrak{n})
\right]_\mathfrak{p}
-ts
\left[
X_\mathfrak{n},
L^*\Lambda_f(X_\mathfrak{p}-tLX_\mathfrak{n})
\right]_\mathfrak{p}
=0,
\\[0.8em]
\begin{aligned}
&
\left[
X_\mathfrak{p},
\Lambda_bX_\mathfrak{n}
-
tsL^*\Lambda_f(X_\mathfrak{p}-tLX_\mathfrak{n})
\right]\\
&+
\left[
X_\mathfrak{n},
s\Lambda_f(X_\mathfrak{p}-tLX_\mathfrak{n})
\right]
+
\left[
X_\mathfrak{n},
\Lambda_bX_\mathfrak{n}
-
tsL^*\Lambda_f(X_\mathfrak{p}-tLX_\mathfrak{n})
\right]_\mathfrak{n}=0.
\end{aligned}
\end{array}
\right.
\end{equation*}
The first equation gives
\[
s[X_\mathfrak{p},\Lambda_fX_\mathfrak{p}]_\mathfrak{p}
-ts\left(
[X_\mathfrak{p},\Lambda_fLX_\mathfrak{n}]_\mathfrak{p}
+
[X_\mathfrak{n},L^*\Lambda_fX_\mathfrak{p}]_\mathfrak{p}
\right)
+t^2s[X_\mathfrak{n},L^*\Lambda_fLX_\mathfrak{n}]_\mathfrak{p}=0.
\]
Since this holds for all $t\in\mathbb R$ and $s>0,$ we obtain
\[
[X_\mathfrak{p},\Lambda_fX_\mathfrak{p}]_\mathfrak{p}=
[X_\mathfrak{p},\Lambda_fLX_\mathfrak{n}]_\mathfrak{p}
+
[X_\mathfrak{n},L^*\Lambda_fX_\mathfrak{p}]_\mathfrak{p}=[X_\mathfrak{n},L^*\Lambda_fLX_\mathfrak{n}]_\mathfrak{p}=0.
\]
On the other hand, the second equation gives
\[
\begin{aligned}
&[X_\mathfrak{p},\Lambda_bX_\mathfrak{n}]
+
[X_\mathfrak{n},\Lambda_bX_\mathfrak{n}]_\mathfrak{n}
+
s[X_\mathfrak{n},\Lambda_fX_\mathfrak{p}]\\
&-ts\left(
[X_\mathfrak{p},L^*\Lambda_fX_\mathfrak{p}]
+
[X_\mathfrak{n},\Lambda_fLX_\mathfrak{n}]
+
[X_\mathfrak{n},L^*\Lambda_fX_\mathfrak{p}]_\mathfrak{n}
\right)\\
&+
t^2s\left(
[X_\mathfrak{p},L^*\Lambda_fLX_\mathfrak{n}]
+
[X_\mathfrak{n},L^*\Lambda_fLX_\mathfrak{n}]_\mathfrak{n}
\right)=0.
\end{aligned}
\]
Since this identity holds for all $t\in\mathbb R$ and $s>0,$ the coefficients vanish. Hence
\begin{align*}
&[X_\mathfrak{n},\Lambda_fX_\mathfrak{p}]=0\\
&[X_\mathfrak{p},\Lambda_bX_\mathfrak{n}]
+
[X_\mathfrak{n},\Lambda_bX_\mathfrak{n}]_\mathfrak{n}=0,\\
&[X_\mathfrak{p},L^*\Lambda_fX_\mathfrak{p}]
+
[X_\mathfrak{n},\Lambda_fLX_\mathfrak{n}]
+
[X_\mathfrak{n},L^*\Lambda_fX_\mathfrak{p}]_\mathfrak{n}=0,\\
&[X_\mathfrak{p},L^*\Lambda_fLX_\mathfrak{n}]
+
[X_\mathfrak{n},L^*\Lambda_fLX_\mathfrak{n}]_\mathfrak{n}=0.
\end{align*}
Therefore, $X$ satisfies \eqref{7:equations}.\\

Conversely, assume that $X$ satisfies \eqref{7:equations}. Let $g$ be a $G$-invariant metric on $G/K$ for which $\pi:(G/K,g)\longrightarrow (G/H,g_b)$ is a Riemannian submersion. By Corollary~\ref{corollary:compact:general:metric:operator}, the metric operator of $g$ is determined by some pair $(\Lambda_f,L)$. For this pair, the equations in \eqref{7:equations} imply \eqref{geodesic:pi-lemma:compact}. Hence, by Theorem~\ref{theorem:geodesic:criterion:fibration}, $X$ is a geodesic vector on $(G/K,g)$. Since $g$ was arbitrary, $X$ is a $\pi$-equigeodesic vector with respect to $g_b$.\\

Finally, if $\operatorname{Hom}_K(\mathfrak n,\mathfrak p)=0,$ then $L=0$. Substituting $L=0$ in \eqref{7:equations} yields \eqref{3:equations}. The proof is complete.
\end{proof}

\begin{remark}\label{remark:consequences:criterion}
In the setting of Theorem~\ref{theorem:pi-equigeodesic:criterion}, we have the following immediate consequences:

\begin{itemize}
\item[$a)$] If $X\in\mathfrak{p}\oplus\mathfrak{n}$ is a $\pi$-equigeodesic vector with respect to $g_b$, then $
[X_\mathfrak{p},\Lambda_fX_\mathfrak{p}]_\mathfrak{p}=0$ for every metric operator $\Lambda_f:\mathfrak p\to\mathfrak p$ of an $H$-invariant metric on $H/K$. Hence, the fiber component $X_\mathfrak{p}$ is an equigeodesic vector on the fiber $H/K$.

\item[$b)$] Let $X_\mathfrak{p}\in\mathfrak p$ be an equigeodesic vector on the fiber $H/K$. Then $X_\mathfrak{p}$, viewed as a vector in $\mathfrak p\oplus\mathfrak n$, is a $\pi$-equigeodesic vector with respect to $g_b$ if and only if
\[
[X_\mathfrak{p},L^*\Lambda_fX_\mathfrak{p}]=0
\]
for every $L\in\operatorname{Hom}_K(\mathfrak n,\mathfrak p)$ and every metric operator $\Lambda_f:\mathfrak p\to\mathfrak p$ of an $H$-invariant metric on $H/K$. Notice that this condition does not involve the metric operator $\Lambda_b$ of the metric on the base. Therefore, if it holds for one $G$-invariant metric $g_b$ on $G/H$, then it holds for every $G$-invariant metric $g_b$ on $G/H$. In particular, if $\operatorname{Hom}_K(\mathfrak n,\mathfrak p)=0$, then every equigeodesic vector on the fiber $H/K$, viewed as a vector in $\mathfrak p\oplus\mathfrak n$, is a $\pi$-equigeodesic vector with respect to any $G$-invariant metric $g_b$ on $G/H$.
\end{itemize}
\end{remark}

\section{Applications}
\subsection{Flag manifolds} Let $n\geq 5,$ let $s\geq 3$ and let $d_1,...,d_s\in\mathbb{N}$ be such that $n=d_1+\cdots+d_s.$ We consider one of the triples $K\subseteq H\subseteq G$ in the following table:
\[
\renewcommand{\arraystretch}{1.5}
\begin{array}{|c|c|c|}
\hline
G & H & K\\
\hline
\operatorname{SO}(n)
&
\operatorname{S}(\operatorname{O}(d_1+d_2)\times \operatorname{O}(d_3)\times\cdots\times \operatorname{O}(d_s))
&
\operatorname{S}(\operatorname{O}(d_1)\times\cdots\times \operatorname{O}(d_s))
\\
\hline
\operatorname{SU}(n)
&
\operatorname{S}(\operatorname{U}(d_1+d_2)\times \operatorname{U}(d_3)\times\cdots\times \operatorname{U}(d_s))
&
\operatorname{S}(\operatorname{U}(d_1)\times\cdots\times \operatorname{U}(d_s))
\\
\hline
\operatorname{Sp}(n)
&
\operatorname{Sp}(d_1+d_2)\times \operatorname{Sp}(d_3)\times\cdots\times \operatorname{Sp}(d_s)
&
\operatorname{Sp}(d_1)\times\cdots\times \operatorname{Sp}(d_s)
\\
\hline
\end{array}
\]
Then we have a homogeneous fibration
\[
H/K\longrightarrow G/K\xlongrightarrow{\pi}G/H.
\]
The Lie algebras of $\operatorname{SO}(n),$ $\mathrm{SU}(n)$ and $\operatorname{Sp}(n)$ can be realized, respectively, as 
\begin{align*}
        &\mathfrak{so}(n)=\{X\in M_n(\mathbb{R}):X+X^*=0\},\\
        &\mathfrak{su}(n)=\{X\in M_n(\mathbb{C}):X+X^*=0\ \textnormal{and}\ \operatorname{Tr}(X)=0\}\ \textnormal{and}\\
        &\mathfrak{sp}(n)=\{X\in M_n(\mathbb{H}):X+X^*=0\},
\end{align*}
where $X^*$ denotes the transpose, the conjugate transpose or the quaternionic conjugate transpose according to whether $X$ belongs to $\mathfrak{so}(n),$ $\mathfrak{su}(n)$ or $\mathfrak{sp}(n),$ respectively. Let 
\[
\mathbb{F}=
\begin{cases}
    \mathbb{R}, &\textnormal{if}\ G=\operatorname{SO}(n),\\
    \mathbb{C}, &\textnormal{if}\ G=\operatorname{SU}(n),\\
    \mathbb{H}, &\textnormal{if}\ G=\operatorname{Sp}(n).
\end{cases}
\]
Given $X\in M_n(\mathbb{F})$, we write $X=(X_{ij})$ as a block matrix, where $X_{ij}\in M_{d_i\times d_j}(\mathbb{F})$ for all $i,j\in\{1,\ldots,s\}$. We fix the $\operatorname{Ad}(G)$-invariant inner product
\begin{equation}\label{inner:product:example:1}
(X,Y):=-\frac{1}{2}\operatorname{Re}\operatorname{Tr}(XY),
\ X,Y\in\mathfrak{g}.
\end{equation}
For $1\le i<j\le s$, define
\[
\mathfrak{m}_{ij}:=\{X\in M_n(\mathbb{F}):X_{k\ell}=0,\ (k,\ell)\notin\{(i,j),(j,i)\}\ \textnormal{and}\ X_{ji}=-X_{ij}^*\}.
\]
With respect to the inner product in \eqref{inner:product:example:1}, we obtain the orthogonal decompositions
\[
\mathfrak{g}=\mathfrak{h}\oplus \mathfrak{n},
\ \mathfrak{h}=\mathfrak{k}\oplus\mathfrak{p},
\]
where
\[
\mathfrak{p}=\mathfrak{m}_{12},
\ \mathfrak{n}=\bigoplus\limits_{\begin{subarray}{c}1\leq i<j\leq s\\(i,j)\neq (1,2)\end{subarray}}\mathfrak{m}_{ij}.
\]
In this case, we have $\operatorname{Hom}_K(\mathfrak n,\mathfrak p)=0.$ Note that the spaces $\mathfrak{m}_{ij}$, $(i,j)\neq (1,2)$, in the decomposition
of $\mathfrak{n}$ are not necessarily $\operatorname{Ad}(H)$-invariant. In fact,
the pairwise non-equivalent irreducible $\operatorname{Ad}(H)$-invariant
subspaces of $\mathfrak{n}$ are
\begin{align*}
&\mathfrak{m}_{1r}\oplus\mathfrak{m}_{2r},\ 3\leq r\leq s,\\
&\mathfrak{m}_{ij},\ 3\leq i<j\leq s.
\end{align*}
Therefore, any metric operator $\Lambda_b:\mathfrak{n}\to\mathfrak{n}$ on
$G/H$ has the form
\begin{equation*}\label{lambda_b:first:example}
\Lambda_b=
\sum\limits_{r=3}^s\mu_{r}\operatorname{Id}_{\mathfrak{m}_{1r}\oplus\mathfrak{m}_{2r}}
+
\sum\limits_{3\leq i<j\leq s}\mu_{ij}\operatorname{Id}_{\mathfrak{m}_{ij}},
\end{equation*}
with $\mu_r>0$, $r=3,\ldots,s$, and $\mu_{ij}>0$, $3\leq i<j\leq s$.
On the other hand, $\mathfrak{m}_{12}$ is $\operatorname{Ad}(K)$-irreducible,
so any metric operator $\Lambda_f:\mathfrak{p}\to\mathfrak{p}$ on $H/K$ has
the form
\[
\Lambda_f=\lambda \operatorname{Id}_{\mathfrak{m}_{12}},
\ \lambda >0.
\]
Using the relations \eqref{3:equations}, we obtain that
$X\in\mathfrak{p}\oplus\mathfrak{n}$ is $\pi$-equigeodesic with respect to
$\Lambda_b$ if and only if
\begin{align}
&X_{12}X_{2r}=0,\ r=3,...,s, \label{flag:condition:X12X2r}\\
&X_{12}^*X_{1r}=0,\ r=3,\ldots,s, \label{flag:condition:X12starX1r}\\
&\sum\limits_{k=3}^{r-1}(\mu_{kr}-\mu_k)X_{ik}X_{kr}
+
\sum\limits_{k=r+1}^{s}(\mu_k-\mu_{rk})X_{ik}X_{rk}^*=0,
\ i=1,2,\ r=3,\ldots,s, \label{flag:condition:ir}\\
&\begin{aligned}
&
(\mu_i-\mu_j)(X_{1i}^*X_{1j}+X_{2i}^*X_{2j})
+
\sum\limits_{k=3}^{i-1}(\mu_{ki}-\mu_{kj})X_{ki}^*X_{kj}
\\
&+
\sum\limits_{k=i+1}^{j-1}(\mu_{kj}-\mu_{ik})X_{ik}X_{kj}
+
\sum\limits_{k=j+1}^s(\mu_{ik}-\mu_{jk})X_{ik}X_{jk}^*=0,\ 3\leq i<j\leq s.
\end{aligned} \label{flag:condition:ij}
\end{align}
\begin{remark} It is worth making the following observations:
\begin{itemize}
    \item[$a)$] Equations \eqref{flag:condition:X12X2r} and
    \eqref{flag:condition:X12starX1r} are equivalent to $[X_{\mathfrak{p}},X_{\mathfrak{n}}]=0$ while equations \eqref{flag:condition:ir} and \eqref{flag:condition:ij} are equivalent to $[X_{\mathfrak{n}},\Lambda_bX_{\mathfrak{n}}]_{\mathfrak{n}}=0,$ which means that $X_{\mathfrak{n}}$ is a geodesic vector on $(G/H,\Lambda_b).$ Since $\Lambda_f$ is a multiple of the identity on $\mathfrak{p},$ every
    vector in $\mathfrak{p}$ is equigeodesic on the fiber $H/K.$ Hence,
    in this example, a vector $X\in\mathfrak{p}\oplus\mathfrak{n}$ is
    $\pi$-equigeodesic with respect to $\Lambda_b$ if and only if
    $X_{\mathfrak{n}}$ is a geodesic vector on $(G/H,\Lambda_b)$ and $[X_{\mathfrak p},X_{\mathfrak n}]=0.$ In other words, the criterion produces $\pi$-equigeodesic vectors on $G/K$ by combining an equigeodesic vector on the fiber with a geodesic vector on the base, under the additional condition that the two components commute.

        \item[$b)$] A special case is obtained when $\Lambda_b$ is associated with the normal metric on $G/H$ induced by the inner product \eqref{inner:product:example:1}. In this case, all the parameters $\mu_r$ and $\mu_{ij}$ are equal. Hence \eqref{flag:condition:ir} and \eqref{flag:condition:ij} vanish identically and the equations \eqref{flag:condition:X12X2r}-\eqref{flag:condition:ij} reduce to 
        \[
        X_{12}X_{2r}=0,\ X_{12}^*X_{1r}=0,\ r=3,...,s.
        \]
        Thus, for any choice of $X_{12}$, the matrices $X_{2r}$ may be chosen with columns in $\ker X_{12}$, while the matrices $X_{1r}$ may be chosen with columns in $\ker X_{12}^*$. The components $X_{ij}$, $3\le i<j\le s$, are arbitrary in this case. Therefore, whenever these kernels are non-zero, we have $\pi$-equigeodesic vectors with several non-zero components. 
\end{itemize}
\end{remark}
\subsection{Flags of $\operatorname{SO}(4)$}
In this subsection, we consider the Lie groups
\begin{align*}
    &G=\operatorname{SO}(4),\\
    &H=\operatorname{S}(\operatorname{O}(2)\times \operatorname{O}(1)\times \operatorname{O}(1))\ \textnormal{and}\\
    &K=\operatorname{S}(\operatorname{O}(1)\times\operatorname{O}(1)\times \operatorname{O}(1)\times \operatorname{O}(1)).
\end{align*}
For $1\leq i<j\leq 4,$ let $\mathfrak{m}_{ij}:=\operatorname{span}\{A_{ij}\},$ where $A_{ij}$ is the skew-symmetric $4\times 4$ matrix with $1$ in the $(i,j)$-entry, $-1$ in the $(j,i)$-entry and zero elsewhere. The Lie algebras of $G,$ $H$ and $K$ are $\mathfrak{g}=\mathfrak{so}(4)=\{X\in M_4(\mathbb{R}):X+X^*=0\},\ \mathfrak{h}=\mathfrak{m}_{12}\ \textnormal{and}\ \mathfrak{k}=\{0\},$ respectively. With respect to the $\operatorname{Ad}(G)$-invariant inner product 
\[
(X,Y):=-\frac{1}{2}\operatorname{Tr}(XY),\ X,Y\in\mathfrak{so}(4),
\]
we have the orthogonal decompositions
\[
\mathfrak{g}=\mathfrak{h}\oplus\mathfrak{n},\ \mathfrak{h}=\mathfrak{k}\oplus\mathfrak{p},
\]
where 
\[
\mathfrak{p}=\mathfrak{m}_{12},\ \mathfrak{n}=\mathfrak{m}_{34}\oplus(\mathfrak{m}_{13}\oplus\mathfrak{m}_{23})\oplus(\mathfrak{m}_{14}\oplus\mathfrak{m}_{24}).
\]
We have that $\mathfrak{p}$ is $\operatorname{Ad}(K)$-invariant and irreducible, while the spaces $\mathfrak{m}_{34},\mathfrak{m}_{13}\oplus\mathfrak{m}_{23}$ and $\mathfrak{m}_{14}\oplus\mathfrak{m}_{24}$ are $\operatorname{Ad}(H)$-invariant and irreducible. Up to this point, the setting is analogous to the families of flag manifolds described in the previous section; however, in this case, the $H$-modules $\mathfrak{m}_{13}\oplus\mathfrak{m}_{23}$ and $\mathfrak{m}_{14}\oplus\mathfrak{m}_{24}$ are equivalent (see \cite[Section 5.1]{PSM} or \cite[Section 3.1]{GGN}) and $\operatorname{Hom}_K(\mathfrak{n},\mathfrak{p})\neq 0.$ In fact,
\[
\operatorname{Hom}_K(\mathfrak{n},\mathfrak{p})\cong\bigoplus\limits_{\begin{subarray}{c}1\leq i<j\leq 4\\(i,j)\neq (1,2)\end{subarray}}\operatorname{Hom}_{K}(\mathfrak{m}_{ij},\mathfrak{m}_{12})\cong \operatorname{Hom}_K(\mathfrak{m}_{34},\mathfrak{m}_{12}),
\]
where the last isomorphism holds because $\mathfrak{m}_{34}$ is the only $K$-module equivalent to $\mathfrak{m}_{12}$ \cite[Section 5.1]{PSM}. Therefore, each $L\in\operatorname{Hom}_K(\mathfrak{n},\mathfrak{p})$ is of the form
\[
L\left(\sum\limits_{\begin{subarray}{c}1\leq i<j\leq 4\\(i,j)\neq(1,2)\end{subarray}}x_{ij}A_{ij}\right)=\ell x_{34}A_{12},
\]
for some $\ell\in\mathbb{R}.$ By \cite[Proposition~3.3]{GGN}, every metric operator $\Lambda_b:\mathfrak{n}\to\mathfrak{n}$ has the form
\begin{align*}
\Lambda_b\left(\sum\limits_{\begin{subarray}{c}1\leq i<j\leq 4\\(i,j)\neq(1,2)\end{subarray}}x_{ij}A_{ij}\right)=&\,\mu_{3}(x_{13}A_{13}+x_{23}A_{23})+\mu_4(x_{14}A_{14}+x_{24}A_{24})+\mu_{34}x_{34}A_{34}\\
&+b(x_{24}A_{13}+x_{13}A_{24}-x_{14}A_{23}-x_{23}A_{14})
\end{align*}
for some $\mu_3,\mu_4,\mu_{34}>0$ and $b\in\mathbb{R}$ such that $b^2< \mu_3\mu_4.$ Since $\mathfrak{p}$ is one-dimensional, every metric operator $\Lambda_f:\mathfrak{p}\to\mathfrak{p}$ on the fiber has the form 
\[
\Lambda_f(x_{12}A_{12})=\lambda x_{12}A_{12},\ \lambda>0.
\]
For a given $X=\sum\limits_{1\leq i<j\leq 4}x_{ij}A_{ij}$, we have
\[
X_{\mathfrak{p}}=x_{12}A_{12}\ \textnormal{and}\ X_{\mathfrak{n}}=x_{13}A_{13}+x_{23}A_{23}
+x_{14}A_{14}+x_{24}A_{24}
+x_{34}A_{34}.
\]
A direct substitution in \eqref{7:equations} gives that $X$ is $\pi$-equigeodesic with respect to $\Lambda_b$ if and only if 
\begin{equation}\label{pi-equigeodesic:SO(4)}
\begin{aligned}
&x_{12}x_{13}=x_{12}x_{23}=x_{12}x_{14}=x_{12}x_{24}=0,\\
&x_{34}x_{13}=x_{34}x_{23}=x_{34}x_{14}=x_{34}x_{24}=0,\\
&(\mu_3-\mu_4)(x_{13}x_{14}+x_{23}x_{24})=0.
\end{aligned}
\end{equation}
\begin{remark}
The previous equations show that the notion of $\pi$-equigeodesic vector is strictly weaker than the notion of equigeodesic vector on $G/K$. For instance, the vector $X=A_{13}+A_{23}$ satisfies \eqref{pi-equigeodesic:SO(4)} and hence is $\pi$-equigeodesic with respect to any $\Lambda_b.$ However, it is not equigeodesic on $G/K$. In fact, there exists a $K$-invariant metric operator $\Lambda:\mathfrak{so}(4)\to\mathfrak{so}(4)$ such that
\[
\Lambda A_{13}=\alpha A_{13},\ \Lambda A_{23}=\beta A_{23},
\]
with $\alpha,\beta>0$ and $\alpha\neq\beta.$ Since
\[
[X,\Lambda X]=[A_{13}+A_{23},\alpha A_{13}+\beta A_{23}]=(\alpha-\beta)A_{12}\neq0,
\]
then $X$ is not equigeodesic.
\end{remark}

\subsection{Ledger-Obata spaces} Let $S$ be a compact connected simple Lie group with Lie algebra
$\mathfrak{s}$ and let $r\geq 2$. We consider
\begin{align*}
&G=S^{r+1},\\
&H=\{(a_1,...,a_r,a_r)\in S^{r+1}:a_i\in S\}, \textnormal{and}\\
&K=\Delta S=\{(a,...,a)\in S^{r+1}:a\in S\}.
\end{align*}

Then $K\subseteq H\subseteq G$ and we have a homogeneous fibration
\[
H/K\longrightarrow G/K\xlongrightarrow{\pi}G/H.
\]
Here $G/K=S^{r+1}/\Delta S$ is a Ledger-Obata space (see \cite[Section~4]{LO}), $H/K\simeq S^r/\Delta S$ and $G/H\simeq S.$ If $Q$ is an $\operatorname{Ad}(S)$-invariant inner product on
$\mathfrak{s},$ then
\[
((A_1,...,A_{r+1}),(B_1,...,B_{r+1}))
:=\sum_{i=1}^{r+1}Q(A_i,B_i)
\]
is an $\operatorname{Ad}(G)$-invariant inner product on $\mathfrak{g}=\mathfrak{s}^{r+1}.$
The Lie algebras of $G$, $H$ and $K$ are given by $\mathfrak{g}=\mathfrak{s}^{r+1},\ \mathfrak{h}=\{(A_1,...,A_r,A_r)\in\mathfrak{s}^{r+1}:A_i\in\mathfrak{s}\}$ and $\mathfrak{k}=\{(A,...,A)\in\mathfrak{s}^{r+1}:A\in\mathfrak{s}\},$ respectively. With respect to the inner product above, we have the orthogonal decompositions
\[
\mathfrak g=\mathfrak h\oplus\mathfrak n,\ \mathfrak h=\mathfrak k\oplus\mathfrak p,
\]
where
\begin{align*}
    \mathfrak{n}=\{N_B:=(0,...,0,B,-B)\in \mathfrak{s}^{r+1}:B\in\mathfrak{s}\}.
\end{align*}
For $i=1,...,r-1$ and $A\in\mathfrak{s}$, let
\[
P_i(A)=(0,...,0,A,0,...,0,-A/2,-A/2)\in\mathfrak{s}^{r+1},
\]
where $A$ is in the $i$-th entry. Then
\[
\mathfrak{p}=
\left\{
\sum_{i=1}^{r-1}P_i(A_i)\in\mathfrak{s}^{r+1}:A_i\in\mathfrak{s}
\right\}.
\]
Under the identification $N_B\leftrightarrow B,$ the $H$-module $\mathfrak{n}$ is equivalent to the adjoint representation of $S$ on $\mathfrak{s}.$ In fact, if $h=(a_1,...,a_r,a_r)\in H$, then
\[
\operatorname{Ad}(h)N_B=N_{\operatorname{Ad}(a_r)B}.
\]
Since $a_r\in S$ can be arbitrary and $\mathfrak{s}$ is simple, it follows that $\mathfrak{n}$ is irreducible as an $H$-module. On the other hand, $\mathfrak{p}$ decomposes as
\[
\mathfrak{p}=P_1(\mathfrak s)\oplus\cdots\oplus P_{r-1}(\mathfrak s),
\]
where each summand is irreducible and, as a $K$-module, is equivalent to $\mathfrak{n}.$ More precisely, if $k=(a,...,a)\in K$, then
\[
\operatorname{Ad}(k)N_B=N_{\operatorname{Ad}(a)B}\ \textnormal{and}\ \operatorname{Ad}(k)P_i(A)=P_i(\operatorname{Ad}(a)A).
\]
Thus, the map $N_B\mapsto P_i(B)$ is a $K$-equivariant isomorphism from $\mathfrak{n}$ onto $P_i(\mathfrak{s}).$ Therefore, since the adjoint representation of $S$ on $\mathfrak{s}$ is orthogonal,
\[
\operatorname{Hom}_K(\mathfrak n,P_i(\mathfrak s))\cong\mathbb R,
\ i=1,...,r-1.
\]
Therefore,
\[
\operatorname{Hom}_K(\mathfrak n,\mathfrak p)
\cong
\bigoplus\limits_{i=1}^{r-1}
\operatorname{Hom}_K(\mathfrak n,P_i(\mathfrak s))
\cong
\mathbb R^{r-1}.
\]
This means that each $L\in\operatorname{Hom}_K(\mathfrak n,\mathfrak p)$ is of the form
\[
L(N_B)=\sum_{i=1}^{r-1}\ell_iP_i(B),\ \ell_1,...,\ell_{r-1}\in\mathbb{R}.
\]
The irreducibility of $\mathfrak{n}$ as an $H$-module implies that every metric operator $\Lambda_b:\mathfrak{n}\to\mathfrak{n}$ on the base has the form
\[
\Lambda_bN_B=\mu N_B,\ \mu>0.
\]
Since $\mathfrak{p}$ is a direct sum of $r-1$ pairwise equivalent irreducible $K$-modules, all metric operators $\Lambda_f:\mathfrak{p}\to\mathfrak{p}$ are of the form
\[
\Lambda_f\left(\sum_{i=1}^{r-1}P_i(A_i)\right)=\sum_{i,j=1}^{r-1}a_{ij}P_i(A_j),
\]
where the coefficients $a_{ij}$ are such that the operator defined above is self-adjoint and positive with respect to the restriction of $(\cdot,\cdot)$ to $\mathfrak{p}.$\\

Before characterizing the $\pi$-equigeodesic vectors on $G/K$ with respect to $\Lambda_b,$ we list the following useful relations, which hold for all $A,B,C\in\mathfrak{s}:$
\begin{align}
    &[P_i(A),N_B]=\frac{1}{2}N_{[B,A]},\label{[P_i,N]}\\
    &[N_B,N_C]=(0,...,0,[B,C],[B,C])\in\mathfrak{h},\label{[N,N]}\\
    &[N_B,N_C]_{\mathfrak{p}}=0,\label{[N,N]:2}\\
    &[A,B]=0\Longrightarrow [P_i(A),P_j(B)]=0,\ i,j=1,...,r-1. \label{[P_i,P_j]}
\end{align}
\begin{proposition}
Let $X=X_{\mathfrak{p}}+X_{\mathfrak{n}}\in\mathfrak{p}\oplus\mathfrak{n},$ where
\[
X_{\mathfrak p}=\sum_{i=1}^{r-1}P_i(A_i),\  X_{\mathfrak n}=N_B,
\]
and $A_1,...,A_{r-1},B\in\mathfrak{s}.$ Then $X$ is a $\pi$-equigeodesic vector with respect to any metric operator $\Lambda_b$ on $G/H$ if and only if
\begin{align}\label{pi-equigeodesic:Ledger-Obata}
\begin{split}
&[A_i,A_j]=0,\ 1\leq i<j\leq r-1,\\
&[A_k,B]=0,\ k=1,...,r-1.
\end{split}
\end{align}
\end{proposition}

\begin{proof} Assume first that equations \eqref{pi-equigeodesic:Ledger-Obata} hold. Let $\Lambda_f$ be any metric operator on $\mathfrak{p}.$ Then
\[
\Lambda_fX_{\mathfrak p}=\sum_{i,j=1}^{r-1}a_{ij}P_i(A_j)=\sum\limits_{i=1}^{r-1}P_i\left(\sum\limits_{j=1}^{r-1}a_{ij}A_j\right)=\sum\limits_{i=1}^{r-1}P_i(C_i),
\]
where  $$C_i=\sum\limits_{j=1}^{r-1}a_{ij}A_j,\ i=1,...,r-1.$$ Since each $C_i$ is a linear combination of $A_1,...,A_{r-1},$ it follows from \eqref{pi-equigeodesic:Ledger-Obata} that
\[
[A_i,C_j]=[B,C_j]=0,\ i,j=1,...,r-1.
\]
Therefore
\begin{align*}
    [X_\mathfrak{p},\Lambda_fX_\mathfrak{p}]&=\sum\limits_{i,j=1}^{r-1}[P_i(A_i),P_j(C_j)]=0,
\end{align*}
where the last equality follows from \eqref{[P_i,P_j]}. Moreover, using \eqref{[P_i,N]}, we obtain
\begin{align*}
    [X_\mathfrak{n},\Lambda_fX_\mathfrak{p}]&=\sum\limits_{i=1}^{r-1}[N_B,P_i(C_i)]=-\frac{1}{2}\sum\limits_{i=1}^{r-1}N_{[B,C_i]}=0.
\end{align*} Now let $L\in\operatorname{Hom}_K(\mathfrak n,\mathfrak p).$ Then
\[
LN_B=\sum_{i=1}^{r-1}\ell_iP_i(B)
\]
for some $\ell_1,...,\ell_{r-1}\in\mathbb R.$ Therefore
\[
\Lambda_fLX_{\mathfrak n}=\sum_{i=1}^{r-1}P_i(D_i),
\]
where $$D_i=\left(\sum\limits_{j=1}^{r-1}a_{ij}\ell_j\right)B,\ i=1,...,r-1.$$ 
We claim that, for a given
\[
Y=\sum\limits_{i=1}^{r-1}P_i(Y_i)\in\mathfrak{p},
\]
we have $L^*Y=N_Z,$ where $Z$ is a linear combination of $Y_1,...,Y_{r-1}.$ Indeed, since $L^*Y\in\mathfrak n,$ there exists $Z\in\mathfrak s$ such that $L^*Y=N_Z.$ This element is determined by
\[
(LN_W,Y)=(N_W,N_Z),\ \textnormal{for all }W\in\mathfrak{s}.
\]
Now, $(N_W,N_Z)=2Q(W,Z)=Q(W,2Z),$ while
\begin{align*}
    (LN_W,Y)&=\left(\sum\limits_{i=1}^{r-1}\ell_iP_i(W),\sum\limits_{j=1}^{r-1}P_j(Y_j)\right)\\[0.5em]
    &=\sum\limits_{i,j=1}^{r-1}\ell_i(P_i(W),P_j(Y_j))\\[0.5em]
    &=\sum\limits_{i=1}^{r-1}\ell_i(P_i(W),P_i(Y_i))+\sum\limits_{\begin{subarray}{c}i,j=1\\i\neq j\end{subarray}}^{r-1}\ell_i(P_i(W),P_j(Y_j))\\[0.5em]
    &=\frac{3}{2}\sum\limits_{i=1}^{r-1}\ell_iQ(W,Y_i)+\frac{1}{2}\sum\limits_{\begin{subarray}{c}i,j=1\\i\neq j\end{subarray}}^{r-1}\ell_iQ(W,Y_j)\\[0.5em]
    &=Q\left(W,\frac{3}{2}\sum\limits_{i=1}^{r-1}\ell_iY_i+\frac{1}{2}\sum\limits_{\begin{subarray}{c}i,j=1\\i\neq j\end{subarray}}^{r-1}\ell_iY_j\right).
\end{align*}
Therefore,
\[
Z=\frac{3}{4}\sum\limits_{i=1}^{r-1}\ell_iY_i+\frac{1}{4}\sum\limits_{\begin{subarray}{c}i,j=1\\i\neq j\end{subarray}}^{r-1}\ell_iY_j,
\]
which is a linear combination of $Y_1,...,Y_{r-1}.$ In particular, $L^*\Lambda_fX_{\mathfrak p}$ is of the form $N_E,$ where $E$ is a real linear combination of $C_1,...,C_{r-1},$ and therefore also of $A_1,...,A_{r-1}.$ Moreover, $L^*\Lambda_fLX_{\mathfrak{n}}$ is of the form $N_F,$ where $F$ is a real linear combination of $D_1,...,D_{r-1}.$ Since the vectors $D_i$ are all  multiples of $B,$ it follows that $F=\eta B$ for some $\eta\in\mathbb{R}.$ Computing the remaining brackets in \eqref{7:equations}, we obtain
\begin{align*}
    &[X_\mathfrak{n},L^*\Lambda_fLX_\mathfrak{n}]_\mathfrak{p}
    =[N_B,N_{\eta B}]_{\mathfrak{p}}=0,\\
    &[X_\mathfrak{p},\Lambda_fLX_\mathfrak{n}]_\mathfrak{p}+[X_\mathfrak{n},L^*\Lambda_fX_\mathfrak{p}]_\mathfrak{p}=\sum\limits_{i,j=1}^{r-1}\left[P_i(A_i),P_j(D_j)\right]_{\mathfrak p}
    +[N_B,N_E]_{\mathfrak p}=0,\\
    &[X_\mathfrak{p},\Lambda_bX_\mathfrak{n}]+[X_\mathfrak{n},\Lambda_bX_\mathfrak{n}]_\mathfrak{n}=
    \mu\sum\limits_{i=1}^{r-1}[P_i(A_i),N_B]+[N_B,\mu N_B]_\mathfrak n
    =
    \frac{\mu}{2}\sum\limits_{i=1}^{r-1}N_{[B,A_i]}=0,\\
    &[X_\mathfrak{n},\Lambda_fLX_\mathfrak{n}]
    +[X_\mathfrak{p},L^*\Lambda_fX_\mathfrak{p}]
    +[X_\mathfrak{n},L^*\Lambda_fX_\mathfrak{p}]_\mathfrak{n}\\
    &\hspace{2em}=
    \sum\limits_{i=1}^{r-1}\left[N_B,P_i(D_i)\right]
    +\sum\limits_{i=1}^{r-1}\left[P_i(A_i),N_E\right]
    +
    [N_B,N_E]_\mathfrak n\\
    &\hspace{2em}=-\frac{1}{2}\sum\limits_{i=1}^{r-1}N_{[B,D_i]}+\frac{1}{2}\sum\limits_{i=1}^{r-1}N_{[E,A_i]}=0,\\
    &[X_\mathfrak{p},L^*\Lambda_fLX_\mathfrak{n}]
    +[X_\mathfrak{n},L^*\Lambda_fLX_\mathfrak{n}]_\mathfrak{n}\\
    &\hspace{2em}=
    \sum\limits_{i=1}^{r-1}\left[P_i(A_i),N_{\eta B}\right]+[N_B,N_{\eta B}]_\mathfrak n
    =
    \frac{\eta}{2}\sum\limits_{i=1}^{r-1}N_{[B,A_i]}=0.
\end{align*}
Here we used \eqref{[P_i,N]}-\eqref{[P_i,P_j]}, together with the facts that each $D_i$ is a multiple of $B$ and that $E$ is a linear combination of $A_1,...,A_{r-1}.$ Thus $X$ is a $\pi$-equigeodesic vector with respect to any $\Lambda_b.$\\

Conversely, assume that $X$ is a $\pi$-equigeodesic vector with respect to a metric operator $\Lambda_b.$ By
\eqref{7:equations}, for every metric operator $\Lambda_f$ on $\mathfrak p$,
we have in particular
\[
[X_{\mathfrak p},\Lambda_fX_{\mathfrak p}]_{\mathfrak p}=[X_{\mathfrak n},\Lambda_fX_{\mathfrak p}]=0.
\]
We will show that these two conditions force relations \eqref{pi-equigeodesic:Ledger-Obata}. Fix $1\leq i<j\leq r-1$ and define the linear map $T_{ij}:\mathfrak p\to\mathfrak p$ by
\[
T_{ij}\left(\sum_{k=1}^{r-1}P_k(Z_k)\right)=P_i(Z_i-Z_j)-P_j(Z_i-Z_j).
\]
For each $\kappa=(a,...,a)\in K,$ we have 
\begin{align*}
    T_{ij}\left(\operatorname{Ad}(\kappa)\left(\sum_{k=1}^{r-1}P_k(Z_k)\right)\right)&=T_{ij}\left(\sum\limits_{k=1}^{r-1}\operatorname{Ad}(\kappa)\left(P_k(Z_k)\right)\right)\\[0.5em]
    &=T_{ij}\left(\sum\limits_{k=1}^{r-1}P_k\left(\operatorname{Ad}(a)Z_k\right)\right)\\[0.5em]
    &=P_i\left(\operatorname{Ad}(a)(Z_i-Z_j)\right)-P_j\left(\operatorname{Ad}(a)(Z_i-Z_j)\right)\\[0.5em]
    &=\operatorname{Ad}(\kappa)P_i(Z_i-Z_j)-\operatorname{Ad}(\kappa)P_j(Z_i-Z_j)\\[0.5em]
    &=\operatorname{Ad}(\kappa)T_{ij}\left(\sum_{k=1}^{r-1}P_k(Z_k)\right).
\end{align*}
Therefore, $T_{ij}$ is $K$-equivariant. Also, for
$Z=\sum_{k=1}^{r-1}P_k(Z_k), W=\sum_{k=1}^{r-1}P_k(W_k)\in\mathfrak{p},$ one has
\[
(T_{ij}Z,W)=Q(Z_i-Z_j,W_i-W_j)=(Z,T_{ij}W),
\]
and in particular
\[
(T_{ij}Z,Z)=Q(Z_i-Z_j,Z_i-Z_j)\geq 0.
\]
Therefore, the map $\Lambda^{ij}_f:=\operatorname{Id}_{\mathfrak{p}}+T_{ij}$ is a metric operator on $\mathfrak{p}.$ For this $\Lambda^{ij}_f,$ the condition $[X_{\mathfrak{p}},\Lambda^{ij}_fX_{\mathfrak p}]_{\mathfrak{p}}=0$ implies 
\[
[X_{\mathfrak{p}},T_{ij}X_{\mathfrak{p}}]_{\mathfrak{p}}=0.
\]
Now
\begin{align*}
0=[X_{\mathfrak{p}},T_{ij}X_{\mathfrak{p}}]_{\mathfrak{p}}
=&\sum\limits_{k=1}^{r-1}[P_k(A_k),P_i(A_i-A_j)-P_j(A_i-A_j)]_\mathfrak{p}\\[0.5em]
=&(0,...,0,
\underset{\textnormal{$i$-th}}{[A_i,A_i-A_j]},
0,...,0,
\underset{\textnormal{$j$-th}}{-[A_j,A_i-A_j]},
0,...,0)_\mathfrak{p}\\[0.5em]
=&(0,...,0,
\underset{\textnormal{$i$-th}}{-[A_i,A_j]},
0,...,0,
\underset{\textnormal{$j$-th}}{[A_i,A_j]},
0,...,0)_\mathfrak{p}\\[0.5em]
=&P_i(-[A_i,A_j])+P_j([A_i,A_j]).
\end{align*}
Therefore, $[A_i,A_j]=0.$ It remains to prove that $[A_i,B]=0,\ i=1,...,r-1.$
For any
\[
Y=\sum_{j=1}^{r-1}P_j(C_j)\in\mathfrak p,
\]
we have
\[
[N_B,Y]=-\frac{1}{2}N_{\sum_{j=1}^{r-1}[B,C_j]}.
\]
Therefore, the condition $[X_{\mathfrak n},\Lambda_fX_{\mathfrak p}]=0$
is equivalent to
\begin{equation}\label{auxiliar:1}
\left[B,\sum_{j=1}^{r-1}C_j\right]=0,\ \textnormal{where}\
\Lambda_fX_{\mathfrak p}=\sum_{j=1}^{r-1}P_j(C_j).
\end{equation}
For $i\in\{1,...,r-1\},$ define $R_i:\mathfrak{p}\to\mathfrak{p}$ by 
\[
R_i\left(\sum\limits_{j=1}^{r-1}P_j(Z_j)\right):=\frac{r}{r+1}P_i(Z_i)-\frac{1}{r+1}\sum\limits_{\begin{subarray}{c}j=1\\j\neq i\end{subarray}}^{r-1}P_j(Z_i).
\]
Given $\kappa=(a,...,a)\in K$ and $U\in\mathfrak{s},$ we have $\operatorname{Ad}(\kappa)P_j(U)=P_j(\operatorname{Ad}(a)U),$ for all $j\in\{1,...,r-1\}.$ Therefore, each $R_i$ is $K$-equivariant. Moreover, for $Z=\sum_{j=1}^{r-1}P_j(Z_j)$ and $W=\sum_{j=1}^{r-1}P_j(W_j)$ we have 
\[
(R_iZ,W)=Q(Z_i,W_i)=(Z,R_iW).
\]
Thus, $R_i$ is self-adjoint and nonnegative with respect to $(\cdot,\cdot)\big{|}_{\mathfrak{p}\times\mathfrak{p}}.$ This implies that the map $\Lambda_f^i:=\operatorname{Id}_{\mathfrak{p}}+R_i$ is a metric operator on $\mathfrak{p}.$ By \eqref{auxiliar:1}, the conditions $[X_\mathfrak{n},\operatorname{Id}_\mathfrak{p}X_{\mathfrak{p}}]=0$ and $[X_\mathfrak{n},\Lambda_f^iX_{\mathfrak{p}}]=0$ are equivalent to
\[
\left[B,\sum\limits_{j=1}^{r-1}A_j\right]=0
\]
and
\[
\left[B,\sum\limits_{j=1}^{r-1}A_j-\frac{1}{r+1}\sum\limits_{\begin{subarray}{c}j=1\\j\neq i\end{subarray}}^{r-1}A_i+\frac{r}{r+1}A_i\right]=0,
\]
respectively. Subtracting the first equality from the second one, we obtain
\[
0=\left[B,-\frac{1}{r+1}\sum\limits_{\begin{subarray}{c}j=1\\j\neq i\end{subarray}}^{r-1}A_i+\frac{r}{r+1}A_i\right]
=
\left[B,\left(-\frac{r-2}{r+1}+\frac{r}{r+1}\right)A_i\right]
=
\frac{2}{r+1}[B,A_i].
\]
Therefore, $[B,A_i]=0.$ Since $i$ was arbitrary, we get $[B,A_i]=0,\ i=1,...,r-1.$ This completes the proof.
\end{proof}
\begin{remark} Equations \eqref{pi-equigeodesic:Ledger-Obata} show that $\pi$-equigeodesic vectors in this family are obtained by choosing all components in a common abelian subalgebra of $\mathfrak{s}.$ In particular, if $\mathfrak{t}\subseteq\mathfrak{s}$ is a Cartan subalgebra and $A_1,...,A_{r-1},B\in\mathfrak{t},$ then
\[
X=\sum_{i=1}^{r-1}P_i(A_i)+N_B
\]
is a $\pi$-equigeodesic vector with respect to any metric on the base.
\end{remark}

\section*{Acknowledgements}
Lino Grama is partially supported by FAPESP grants no. 2021/04065-6, and CNPq grant no. 306021/2024-2.

\bibliographystyle{apa}

\end{document}